\documentclass[a4paper]{amsart}

\usepackage{color}
\usepackage{enumerate}
\usepackage{graphics}
\usepackage[colorlinks, urlcolor=blue]{hyperref}
\usepackage{mathrsfs}
\usepackage{amssymb}

\newtheorem{theorem}{Theorem}
\newtheorem*{theorem*}{Theorem}
\newtheorem*{problem*}{Problem}
\newtheorem{lemma}[theorem]{Lemma}
\newtheorem{corollary}[theorem]{Corollary}
\newtheorem{definition}[theorem]{Definition}

\newtheorem*{conjecture*}{Krein's Conjecture}

\theoremstyle{remark}

\newtheorem{remark}[theorem]{Remark}

\begin{document}

\title[Peller's problem]{Resolution of Peller's problem concerning
 Koplienko-Neidhardt trace formulae}

\author[C. Coine]{Clement Coine}
\email{clement.coine@univ-fcomte.fr}
\author[C. Le Merdy]{Christian Le Merdy}
\email{clemerdy@univ-fcomte.fr}
\author[D. Potapov]{Denis Potapov}
\email{d.potapov@unsw.edu.au}
\author[F. Sukochev]{Fedor Sukochev}
\email{f.sukochev@unsw.edu.au}
\author[A. Tomskova]{Anna Tomskova}
\email{a.tomskova@unsw.edu.au}
\address{School of Mathematics \& Statistics, University of NSW,
Kensington NSW 2052 AUSTRALIA}
\address{Laboratoire de Math\'{e}matiques, Universit\'{e} de Franche-comt\'{e},
25030 Besan\c{c}on Cedex}

\maketitle

\begin{abstract}
A formula for the norm of a bilinear Schur multiplier acting from the Cartesian product
$\mathcal S^2\times \mathcal S^2$ of two copies of
the Hilbert-Schmidt classes into the trace class
$\mathcal S^1$ is established in terms of linear Schur multipliers acting
on the space $\mathcal S^\infty$ of all compact operators.
Using this formula, we resolve Peller's problem on Koplienko-Neidhardt trace formulae.
Namely, we prove that there exist a twice continuously differentiable function $f$
with a bounded second derivative, a self-adjoint (unbounded) operator $A$
and a self-adjoint operator $B\in \mathcal S^2$ such that
$$
f(A+B)-f(A)-\frac{d}{dt}(f(A+tB))\big\vert_{t=0}\notin \mathcal S^1.
$$
\end{abstract}

\bibliographystyle{short}

\section{Introduction}

Let $\mathcal H$ be a separable complex Hilbert space and let $B(\mathcal H)$ be the space of all bounded
linear operators on $\mathcal H$ equipped with the standard trace ${\rm Tr}.$ Let $\mathcal S^1
=\mathcal S^1(\mathcal H)$
and $\mathcal S^2 = \mathcal S^2(\mathcal H)$ be the trace class and the
Hilbert-Schmidt class in $B(\mathcal H),$ respectively.

In 1953, M. G. Krein \cite{K1} showed that for a self-adjoint (not necessarily bounded) operator $A$ and a
self-adjoint operator  $B\in\mathcal S^1$ there exists a unique function $\xi\in L^1(\mathbb R)$ such that
\begin{equation}\label{KreinsFormula}{\rm Tr}
(f(A+B)-f(A))=\int_\mathbb R f'(t) \xi(t) dt,
\end{equation}
whenever $f$ is from the Wiener class $W_1$, that is $f$ is a function on $\mathbb R$ with Fourier
transform of $f'$ in $L^1(\mathbb R).$

The function $\xi$ above is called Lifshitz-Krein spectral shift function and was firstly
introduced in a special case by I. M. Lifshitz \cite{Lif}.
It plays an important role in Mathematical Physics and in Scattering Theory, where it appears in the
formula of the determinant of scattering matrix (for detailed discussion we refer to \cite{BY} and references therein).

Observe that the right-hand side of $\eqref{KreinsFormula}$ makes sense for every Lipschitz function $f$.
In 1964 M. G. Krein conjectured that the left-hand side of $\eqref{KreinsFormula}$ also makes sense
for every Lipschitz function $f$.  More precisely, Krein's conjecture was the following.

\begin{conjecture*} For any  self-adjoint (not necessarily bounded) operator
$A$, for any self-adjoint operator  $B\in\mathcal S^1$ and for any  Lipschitz function $f$,
\begin{equation}\label{KreinsConjecture}
f(A+B)-f(A)\in \mathcal S^1.
\end{equation}
\end{conjecture*}

The best result concerning the description of the class of functions for which \eqref{KreinsConjecture}
holds is due to V. Peller in \cite{PelHank}, who established that
\eqref{KreinsConjecture}  holds for $f$ belonging to the Besov class
$B^1_{\infty 1}$ (for a definition of this
class, see \cite{PelHank} and references therein).
However \eqref{KreinsConjecture} does not hold even for the absolute value function,
which is obviously the simplest example of a Lipschitz function
(see e.g. \cite{Davies}, \cite{DDPS}).
Moreover, there is an example of a continuously differentiable Lipschitz function $f$
and (bounded) self-adjoint operators $A,B$ with $B\in\mathcal S^1$ such that
\eqref{KreinsConjecture} does not hold. The first such example is due to
Yu. B. Farforovskaya \cite{F6}.

Assume now that $B$ is a self-adjoint operator from the Hilbert-Schmidt class $\mathcal S^2$.
In 1984, L. S. Koplienko, \cite{Ko1}, considered the operator
\begin{equation}\label{SecondTaylor}
f(A+B)-f(A)-\frac{d}{dt}\Big(f(A+tB)\Big)\Big|_{t=0},
\end{equation}
where by $\frac{d}{dt}\Big(f(A+tB)\Big)\Big|_{t=0}$ we denote the
derivative of the map $t\mapsto f(A+tB)$ in the
Hilbert-Schmidt norm. 
He proved that for every fixed self-adjoint
operator $A$ there exists a unique function $\eta\in L^1(\mathbb R)$ such that
\begin{equation}\label{KoplienkoFormula}
{\rm Tr} \Big(f(A+B)-f(A)-\frac{d}{dt}\Big(f(A+tB)\Big)\Big|_{t=0}\Big)
=\int_{\mathbb R} f''(t) \eta(t)dt,
\end{equation}
if $f$ is an arbitrary rational function with poles off $\mathbb R.$

The function $\eta$ is called Koplienko's spectral shift function
(for more information about Koplienko's spectral shift function we
refer to \cite{GPS2008} and references therein).

It is clear that the  right-hand side of \eqref{KoplienkoFormula} makes sense
when $f$ is a twice differentiable function with a bounded second derivative.
The natural question is then to describe the class of all these functions $f$ such that
the left-hand side of \eqref{KoplienkoFormula} is well-defined. Namely,
for which function $f$ does the operator \eqref{SecondTaylor} belong to $\mathcal S^1?$
The best result to date is again due to V. Peller \cite{Peller2005},
who established an affirmative answer under the assumption that
$f$ belongs to the Besov class $B^2_{\infty 1}$.
In the same paper \cite{Peller2005}, V. Peller stated the following problem.

\bigskip\noindent
{\bf Peller's problem.}
\cite[Problem 2]{Peller2005}
{\it Suppose that  $f$ is a twice continuously differentiable function with
a bounded second derivative. Let $A$ be a self-adjoint (possibly unbounded) operator and
let $B$ be a self-adjoint operator from $\mathcal S^2.$ Is it true that
\begin{equation}\label{MainQuestion}
f(A+B)-f(A)-\frac{d}{dt} 
\Big(f(A+tB)\Big)\Big|_{t=0} \in\mathcal S^1?
\end{equation}
}

In \cite[Theorem 4.6]{Peller2005}, the author defined
the operator in \eqref{SecondTaylor}
for all  $f\in B^2_{\infty 1}$ via an approximation process. The precise meaning of 
\eqref{SecondTaylor} in the case of an arbitrary self-adjoint operator $A$ and an 
arbitrary twice continuously differentiable function $f$  may be a subject of independent 
investigation, which is beyond the scope of the present paper. 
However when $A$ is a bounded self-adjoint operator, then the meaning of the operator in 
\eqref{SecondTaylor} is firmly established
(see e.g. \cite{BiSo1, BiSo2, BiSo3, dPS, dPSW}). 
From this it is immediate to define uniquely the operator in \eqref{SecondTaylor} 
in the case when $A$ is given by a direct sum $\oplus_{n=1}^\infty A_n$, where each 
$A_n$ is a bounded self-adjoint operator, and $B=\oplus_{n=1}^\infty B_n$
is a self-adjoint operator from $\mathcal S^2.$

In this paper we answer Peller's question in the negative
(see Section \ref{sec_PellersProblem}). More precisely we present a class of twice
continuously differentiable functions $f$ with a bounded second derivative
and self-adjoint operators $A=\oplus_{n=1}^\infty A_n$  
and $B=\oplus_{n=1}^\infty B_n$ as above, with
$B\in \mathcal S^2$, such that the operator \eqref{SecondTaylor}
does not belong to $\mathcal S^1.$ The operators $A_n$ will be finite rank.

In essence, the construction leading to these counterexamples  is finite-dimensional; 
this construction is presented in Section \ref{sec_FD}. A key component of our proof is Theorem
\ref{main theorem_finite}, which provides
a new general formula of independent interest for the norm of bilinear Schur multipliers
(see Definition \ref{def_bilinearSchur}) from $\mathcal S^2\times \mathcal S^2$
into $\mathcal S^1$, in terms of a special sequence of Schur multipliers on
$\mathcal S^\infty.$ In Section \ref{sec_Connection}
we establish preliminary results and connect
Peller's problem to bilinear Schur multipliers.

\section{Bilinear Schur multipliers on $\mathcal S^2\times \mathcal S^2$}\label{sec_Schur}

We regard elements of $B(\ell^2)$ as infinite matrices in the usual way
and we let $\|\cdot\|_\infty$ denote the uniform norm on this space.
By $\mathcal S^{p}$ we denote the Schatten von Neumann ideal in
$B(\ell^2)$ equipped with the Schatten $p$-norm $\|\cdot\|_{p}$, $1\le
p\le\infty$.

Likewise for any $n\in\mathbb N$, we let $M_n$ denote the space of all
$n\times n$ matrices with entries in $\mathbb C$, equipped with the uniform norm
$\|\cdot\|_\infty$, and we use the notation $\mathcal S^{p}_n$ to denote that
space equipped with the  $p$-norm $\|\cdot\|_{p}$.

We let $E_{ij}$ denote the standard matrix units either on
$B(\ell^2)$ or on $M_n$, for $i,j\geq 1$ of for  $1\leq i,j\leq n$.

Let $1\le p\le\infty$.
A matrix $M=\{m_{ij}\}_{i,j\ge 1}$
with entries in $\mathbb C$ is said to be a {\it (linear) Schur multiplier on $\mathcal S^p$} if
the following action
$$
M(A):=\sum_{i,j\ge 1}m_{ij}a_{ij}E_{ij}, \ \
A=\{a_{ij}\}_{i,j\ge 1}\in \mathcal S^p,
$$
defines a  bounded linear operator on $\mathcal S^p.$

Clearly, for the matrix $M=\{m_{ij}\}_{i,j\ge 1}$ to be a linear 
Schur multiplier on $\mathcal S^p$ it is necessary that
$\sup_{i,j\ge 1}|m_{ij}|<\infty$.
When $p=2$,  this condition is sufficient, that is,
a matrix $M=\{m_{ij}\}_{i,j\ge 1}$ is a linear Schur multiplier on
$\mathcal S^2$ if and only if $\sup_{i,j\ge 1} |m_{ij}|<\infty$. Moreover
$$
\|M:\mathcal S^2\to \mathcal S^2\|\,=\,\sup_{i,j\ge 1}|m_{ij}|
$$
in this case (see e.g. \cite[Proposition~2.1]{A-1982}).

A simple duality argument shows that if $1\leq p,p'\leq\infty$ are conjugate numbers,
then a matrix $M$ is a linear Schur multiplier on
$\mathcal S^p$ if and only if it is a linear Schur multiplier on
$\mathcal S^{p'}$. Moreover the resulting operators have the same norm, that is,
$\Vert M\colon \mathcal S^p\to \mathcal S^p\Vert = \Vert M\colon
\mathcal S^{p'}\to \mathcal S^{p'}\Vert$. Linear Schur multipliers
on either $\mathcal S^1$ or $\mathcal S^\infty$ have the following
description (see e.g. \cite[Theorem 5.1]{PisierBook}
or \cite[Theorem 6.4]{Bennett1977}).

\begin{theorem}\label{Multipliers on S_infty}
A matrix $M=\{m_{ij}\}_{i,j\ge 1}$ is a linear Schur multiplier on
$\mathcal S^\infty$ (equivalently, on $\mathcal S^1$) if and only if
there exist a Hilbert space $E$ and two bounded sequences
$(\xi_i)_{i\geq 1}$ and $(\eta_j)_{j\geq 1}$ in $E$ such that
\begin{equation}\label{Factor}
m_{ij} = \langle \xi_i,\eta_j\rangle,
\qquad i,j\geq 1.
\end{equation}
Moreover
$$
\|M:\mathcal S^\infty\to \mathcal S^\infty\|\,=\,\inf\bigl\{\sup_i\Vert \xi_i\Vert\,
\sup_j\Vert \eta_j\Vert\bigr\},
$$
where the infimum runs over all possible factorizations (\ref{Factor}).
\end{theorem}

Except for the cases $p=1,2,\infty$ mentioned above, there is
no known description of linear Schur multipliers on
$\mathcal S^p$.

The terminology below is adopted from \cite{ER}, where multilinear Schur
products are defined and studied in the context of completely bounded maps.

\begin{definition}\label{def_bilinearSchur}
Let $1\le r\le\infty$.
A three-dimensional matrix $M=\{m_{ikj}\}_{i,k,j\ge 1}$
with entries in $\mathbb C$ is said to be a {\it bilinear
Schur multiplier into $\mathcal S^r$} if
the following action
$$
M(A,B):=\sum_{i,j,k\ge 1} m_{ikj}a_{ik}b_{kj}E_{ij}, \quad
A=\{a_{ij}\}_{i,j\ge 1}, B=\{b_{ij}\}_{i,j\ge 1}\in \mathcal S^2,
$$
defines a bounded bilinear operator from
$\mathcal S^2\times \mathcal S^2$ into $\mathcal S^r.$
\end{definition}

Of course we can define as well a notion of bilinear Schur multiplier from
$\mathcal S^p\times\mathcal S^q$ into $\mathcal S^r$, whenever $1\leq p,q,r\leq \infty$.
The case when $p=q=r=\infty$ is the object of \cite{ER}.
The main aim of this section is to give a criteria when a matrix
$M$ is a bilinear Schur multiplier from
$\mathcal S^2\times \mathcal S^2$ into $\mathcal S^1$
(see Theorems \ref{main theorem_finite}, \ref{main theorem}, and
Corollary \ref{main-factor} below). Before coming to this, we mention another
(easier) case which will used in Section 5.

\begin{lemma}\label{always bounded}
A matrix  $M=\{m_{ikj}\}_{i,k,j\ge 1}$ is a
bilinear Schur multiplier into $\mathcal S^2$ if and only if
$\sup_{i,j,k\ge 1} |m_{ikj}|<\infty.$ Moreover,
$$
\|M:\mathcal
S^2\times \mathcal S^2\to\mathcal S^2\|=\sup_{i,j,k\ge 1}
|m_{ikj}|.
$$
\end{lemma}

\begin{proof} The inequality $\|M:\mathcal
S^2\times \mathcal S^2\to\mathcal S^2\|\le \sup_{i,j,k\ge 1}
|m_{ikj}|$ is achieved by the following computation. Consider
$A=\{a_{ik}\}_{i,k\ge 1}$ and $B=\{b_{kj}\}_{k,j\ge 1}$ in $\mathcal
S^2$. Then applying the Cauchy-Schwarz inequality, we have
\begin{align*}
\|M(A,B)\|_2^2
&
=\Big\|\sum_{i,j,k\ge 1}
m_{ikj}a_{ik}b_{kj}E_{ij}\Big\|_2^2 = \sum_{i,j\ge
1}\Big|\sum_{k\ge 1} m_{ikj}a_{ik}b_{kj}\Big|^2 \\
&
\le \sup_{i,j,k\ge 1} |m_{ikj}|^2\sum_{i,j\ge 1}\Big(\sum_{k\ge 1}
|a_{ik}b_{kj}|\Big)^2\\
&
\le \sup_{i,j,k\ge 1} |m_{ikj}|^2\sum_{i,j\ge
1}\sum_{k\ge 1} |a_{ik}|^2\sum_{k\ge 1}|b_{kj}|^2\\\
&
\leq \sup_{i,j,k\ge 1} |m_{ikj}|^2\|A\|_2^2\|B\|_2^2.
\end{align*}
The converse inequality is obtained from
$$
\|M:\mathcal
S^2\times \mathcal S^2\to\mathcal S^2\|\ge
\|M(E_{ik},E_{kj})\|_{2}=|m_{ikj}|,
$$
taking the supremum over all $i,j,k\ge 1.$
\end{proof}

We now focus on bilinear Schur multipliers into $\mathcal S^1$.
We start with some background on tensor products.
Given any two Banach spaces $X$ and $Y$, we let
$X\otimes Y$ denote their algebraic tensor product.
For every $u\in X\otimes Y$, the projective
tensor norm of $u$ is defined as
$$
\pi(u):=\inf\Bigl\{\sum\limits_{i=1}^m \|x_i\|\|y_i\|: \
u=\sum\limits_{i=1}^m x_i\otimes y_i, \ m\in\mathbb N\Bigr\}.
$$
Then the completion of $X\otimes Y$ equipped with the norm $\pi$ is called
the projective tensor product of $X$ and $Y$ and is denoted
by $X\widehat{\otimes} Y$.

Let $Z$ be another Banach space and let
$B_2(X\times Y, Z)$ denote the space of all bounded
bilinear operators from $X\times Y$ into $Z$, equipped with the
uniform norm. Next let $B(X\widehat{\otimes} Y, Z)$ denote
the Banach space of all bounded linear operators
from $X\widehat{\otimes} Y$ into $Z$, equipped with the uniform norm.  Then
we have an isometric isomorphism
\begin{equation}\label{Proj}
B_2(X\times Y, Z) =  B(X\widehat{\otimes} Y, Z),
\end{equation}
which is given by $T\mapsto \tilde T$, where $\tilde T(x\otimes y)=T(x,y)$  for any $x\in
X$ and $y\in Y$ (see e.g. \cite[Theorem 2.9]{Ryan}).

Let $\mathcal H$ be a Hilbert space and let $\overline{\mathcal H}$
denote its conjugate space.
For any $h_1,h_2$ in $\mathcal H$, we may identify
$\overline{h_1}\otimes h_2$ with the operator
$h\mapsto \langle h,h_1\rangle h_2\,$ from $\mathcal H$ into $\mathcal H$.
This yields an identification of $\overline{\mathcal H}\otimes \mathcal H$ with the space
of finite rank operators on $\mathcal H$, and this
identification extends to an isometric isomorphism
\begin{equation}\label{S1}
\overline{\mathcal H}\widehat{\otimes} \mathcal H\, =\, S^1(\mathcal H),
\end{equation}
see e.g. \cite[p. 837]{Palmer}.

In the sequel, we regard $M_{n^2}$ as the space of
matrices with columns and rows indexed by $\{1,\ldots, n\}^2$. Thus we write
$E_{(i,k),(j,l)}$ for its standard matrix units.
Then we let $M_n\otimes_{\rm min} M_n$ denote the minimal tensor product
of two copies of $M_n$. According to the definition of $\otimes_{\rm min}$
(see e.g. \cite[IV.4.8]{Tak}),
the isomorphism $J_0\colon M_n\otimes_{\rm min}  M_n
\to M_{n^2}$ given by
\begin{equation}\label{JO}
J_0(E_{ij}\otimes E_{kl})= E_{(i,k),(j,l)}, \qquad
1\leq i,j,k,l\leq n,
\end{equation}
is an isometry.

We now give some duality principles. First we recall that
${\mathcal S^1_n}^*$ is isometrically isomorphic to $M_n$ through the duality
pairing
\begin{equation}\label{Dual}
\mathcal S^1_n\times M_n\to {\mathbb C},\qquad (A,B)\mapsto {\rm Tr}\bigl({}^t AB\bigr).
\end{equation}
With this convention (note the use of transposition),
the dual basis of $(E_{ij})_{1\leq i,j\leq n}$
is $(E_{ij})_{1\leq i,j\leq n}$ itself.

Next we let $\gamma$ be the cross norm on $\mathcal S^1_n\otimes \mathcal S^1_n$
such that
\begin{equation}\label{gamma}
\bigl(\mathcal S^1_n\otimes_\gamma \mathcal S^1_n\bigr)^* =
M_n\otimes_{\rm min} M_n,
\end{equation}
through the duality pairing (\ref{Dual}) applied twice. More explicitly,
for any family $(t_{ijkl})_{1\leq i,j,k,l\leq n}$ of complex numbers, we have
\begin{align*}
\gamma\biggl(\sum_{i,j,k,l=1}^n & t_{ijkl} E_{ij}\otimes E_{kl}\biggr)  = \\
& \sup\biggl\{\Bigl|\sum_{i,j,k,l=1}^n t_{ijkl}s_{ijkl}\Big| \, :\,
\Big\|\sum_{i,j,k,l=1}^n s_{ijkl} E_{ij}\otimes
E_{kl}\Big\|_{M_n\otimes_{\rm min} M_n}\le 1\biggr\}.
\end{align*}

\begin{lemma}\label{l2}
The isomorphism $J\colon\mathcal S^2_n\widehat{\otimes}\mathcal S^2_n\to
\mathcal S^1_n\otimes_\gamma \mathcal S^1_n$ given by
$$
J(E_{ik}\otimes E_{jl})=
E_{ij}\otimes E_{kl},\qquad 1\leq i,j,k,l\leq n,
$$
is an isometry.
\end{lemma}

\begin{proof}
According to the equality
$$
\Bigl\Vert \sum_{i,k} c_{ik} E_{ik}\Bigr\Vert_2 = \Bigl(\sum_{i,k}\vert
c_{ik}\vert^2\Bigr)^{\frac12},\qquad c_{ik}\in {\mathbb C},
$$
we can naturally identify $\mathcal S^2_n$ with either $\ell^2_{n^2}$ or its conjugate space.
Then applying the identity (\ref{S1}) with $\mathcal H=\ell^2_{n^2}$,
we obtain that
the mapping $J_1\colon\mathcal S^2_n\widehat{\otimes}\mathcal S^2_n\to\mathcal S^1_{n^2}$ given
by
$$
J_1(E_{ik}\otimes E_{jl}) = E_{(i,k),(j,l)},\qquad
1\leq i,j,k,l\leq n,
$$
is an isometry.

Now let $J_2\colon \mathcal S^1_n\otimes_\gamma \mathcal S^1_n \to \mathcal S^1_{n^2}$
be the isomorphism given
by
$$
J_2(E_{ij}\otimes E_{kl}) = E_{(i,k),(j,l)},\qquad
1\leq i,j,k,l\leq n.
$$
Taking into account the identity (\ref{gamma}), we see
that $J_2^{-1}$ is
the adjoint of $J_0$. Consequently, $J_2^{-1}$ is an isometry. Since
$J=J_2^{-1}J_1$, we deduce that $J$ is an isometry as well.
\end{proof}

We will work with the subspace of $M_n\otimes_{\rm min} M_n$ spanned
by the $E_{rk}\otimes E_{ks}$, for $1\leq r,k,s\leq n$. The next lemma
provides a description of this subspace.
We let $(e_1,\ldots,e_n)$ denote the standard basis of $\ell^\infty_n$.

\bigskip\noindent
\begin{lemma}\label{l3}
The linear mapping $\theta\colon
\ell^\infty_n(M_n)\to M_n\otimes_{\rm min} M_n$
such that
$$
\theta(e_k\otimes E_{rs})= E_{rk}\otimes E_{ks},
\qquad 1\leq k,r,s\leq n,
$$
is an isometry.
\end{lemma}

\begin{proof}
Take $y=\sum_{k=1}^n e_k\otimes y_k\in \ell_n^\infty(M_n),$ where
$y_k=\sum_{r,s=1}^n y_k(r,s)E_{rs}.$ From the definition of
$\theta$ we have
$$
\theta(y)=\sum_{r,s,k=1}^n y_k(r,q)E_{rk}\otimes E_{ks}.
$$
Recall the isometric isomorphism $J_0$ given by (\ref{JO}).
Then
$$
J_0\theta(y)=\sum_{r,s,k=1}^n y_k(r,s)E_{(r,k),(k,s)}.
$$
Let $a=\{a_{rk}\}_{r,k=1}^n, b=\{b_{ls}\}_{l,s=1}^n\in \ell^2_{n^2}.$
Then we have
$$
\bigl\langle J_0\theta(y)b,a\bigr\rangle
= \sum_{r,s,k=1}^n y_k(r,s)\langle E_{(r,k),(k,s)}(b),a \rangle =
\sum_{r,s,k=1}^n y_k(r,s) a_{rk}b_{ks}.
$$
Therefore, using Cauchy-Schwarz, we obtain
\begin{align*}
\bigl\vert \bigl\langle J_0\theta(y)b,a\bigr\rangle\bigr\vert
&
\leq \sum_{k=1}^n\Big|\sum_{r,s=1}^n
y_k(r,s) a_{rk}b_{ks}\Big|   \\
&
\le \sum_{k=1}^n
\|y_k\|\Big(\sum_{r=1}^n |a_{rk}|^2\Big)^{\frac12}
\Big(\sum_{s=1}^n |b_{ks}|^2\Big)^{\frac12} \\
&
\le \max_{1\le k\le
n}\|y_k\| \sum_{k=1}^n \Big(\sum_{r=1}^n |a_{rk}|^2\Big)^{\frac12}
\Big(\sum_{s=1}^n |b_{ks}|^2\Big)^{\frac12}\\
&
\le \max_{1\le k\le n}\|y_k\| \Big(\sum_{k,r=1}^n |a_{rk}|^2\Big)^{\frac12}
\Big(\sum_{k,s=1}^n |b_{ks}|^2\Big)^{\frac12}\\
&
\leq \max_{1\le k\le
n}\|y_k\| \|a\|_2 \|b\|_2.
\end{align*}
It follows that $\|\theta(y)\| \le \max_{1\le k\le n}\|y_{k}\|.$

Now fix $1\le k_0\le n.$ Take arbitrary
$\alpha=\{\alpha_r\}_{r=1}^n$ and
$\beta=\{\beta_s\}_{s=1}^n$ in $\ell^2_n$. Then define
$$
a_{rk} :=\left\{\begin{array}{cl}
\alpha_r,& \text{if} \ \ k=k_0 \\
0 & \text{otherwise}\end{array}\right., \quad
b_{ls} :=\left\{\begin{array}{cl}
\beta_s,& \text{if} \ \ l=k_0 \\
0 & \text{otherwise}\end{array}\right..
$$
Then
$$
\bigl\langle J_0\theta(y)b,a\bigr\rangle=
\langle y_{k_0}(\beta),\alpha\rangle
$$
and moreover,
$\Vert a\Vert _2=\Vert \alpha \Vert _2$,
$\Vert b\Vert _2=\Vert \beta \Vert _2$.
Therefore, we have $\|y_{k_0}\|\le \|\theta(y)\|.$
Hence, $\|\theta(y)\|\ge \max_{1\le k\le n}\|y_{k}\|.$
\end{proof}

The following theorem is the main result of this section.

\begin{theorem}\label{main theorem_finite} Let $n\in\mathbb N.$
Let $M=\{m_{ikj}\}_{i,k,j=1}^n$ be a three-dimensional matrix.
For any $1\leq k\leq n$, let $M(k)$ be the (classical) matrix given by
$M(k)=\{m_{ikj}\}_{i,j=1}^n$.
Then
$$
\bigl\Vert M:\mathcal S^2_n\times \mathcal S^2_n\to \mathcal S^1_n\bigr\Vert
=\sup_{1\le k\le n}\bigl\Vert  M(k): M_n\to M_n\bigr\Vert.
$$
\end{theorem}

\begin{proof}
According to the isometric identity (\ref{Proj}),
the bilinear map $M:\mathcal S^2_n\times \mathcal S^2_n\to \mathcal S^1_n$
induces a linear map
$\widetilde{M}\colon\mathcal S^2_n\widehat{\otimes}\mathcal S^2_n\to \mathcal S^1_n$
with $\Vert M\Vert  = \Vert\widetilde{M}\Vert$. Consider
$$
T_M = (\widetilde{M}J^{-1})^*\colon M_n\to M_n\otimes_{\rm min} M_n,
$$
where $J$ is given by Lemma \ref{l2}. The latter implies that
\begin{equation}\label{T}
\Vert T_M\Vert = \bigl\Vert M\colon \mathcal S^2_n\times\mathcal  S^2_n \to\mathcal  S^1_n\bigr\Vert.
\end{equation}
For any $1\leq r,s \leq n$, we have
\begin{align*}
\bigl\langle T_M(E_{rs}), E_{ij}\otimes E_{kl}\bigr\rangle & =
\bigl\langle E_{rs},
\widetilde{M}J^{-1}(E_{ij}\otimes E_{kl})\bigr\rangle\\
& =
\bigl\langle E_{rs},
\widetilde{M}(E_{ik}\otimes E_{jl})\bigr\rangle\\
& =
\left\{\begin{array}{cl}
m_{ikl}\langle E_{rs}, E_{il}\rangle,& \text{if} \ \ k=j \\
0 & \text{otherwise}\end{array}\right.\\
&
=\left\{\begin{array}{cl}
m_{ikl},& \text{if} \ \ k=j, \ r=i, \ s=l \\ 0 &
\text{otherwise}\end{array}\right.,
\end{align*}
for all $1\leq i,j,k,l\leq n$.
Hence
$$
T_M(E_{rs}) = \sum_{k=1}^n m_{rks} E_{rk}\otimes E_{ks}.
$$
This shows that $T_M$ maps into the range of the operator $\theta$
introduced in Lemma \ref{l3} and that
$$
T_M(E_{rs})=\sum_{k=1}^n m_{rks}\,\theta(e_k\otimes E_{rs}).
$$
By linearity this implies that for any $C\in M_n$,
$$
T_M(C) = \theta\biggl(\sum_{k=1}^n e_k\otimes [M(k)](C)\biggr).
$$
Appyling Lemma \ref{l3}, we deduce that
$$
\Vert T_M(C)\Vert = \max_k \bigl\Vert [M(k)](C)\bigr\Vert,\qquad C\in M_n.
$$
From this identity we obtain that $\Vert T_M\Vert = \max_k \Vert M(k)\Vert$. Combining with
(\ref{T}) we obtain the desired
identity $\Vert M\Vert = \max_k \Vert M(k)\Vert $.
\end{proof}

For the sake of completeness we give an infinite dimensional
version of the previous theorem.

\begin{theorem}\label{main theorem}
A three-dimensional matrix $M=\{m_{ikj}\}_{i,k,j\ge 1}$
is a bilinear Schur multiplier into
$\mathcal S^1$ if and only if the matrix
$M(k)=\{m_{ikj}\}_{i,j\ge 1}$ is a linear Schur multiplier on $\mathcal
S^\infty$ for every $k\ge 1$ and $\sup_{k\ge 1}\|
M(k):\mathcal
S^\infty\to \mathcal
S^\infty\|<\infty.$ Moreover,
$$
\bigl\Vert M:\mathcal S^2\times \mathcal S^2\to
\mathcal S^1\bigr\Vert
=\sup_{k\ge 1}\,\bigl\Vert  M(k):\mathcal
S^\infty\to\mathcal S^\infty\bigr\Vert
$$
in this case.
\end{theorem}

\begin{proof}
Consider a three-dimensional matrix $M=\{m_{ikj}\}_{i,k,j\ge 1}$ and set
$M(k)=\{m_{ikj}\}_{i,j\ge 1}$. For any $n\geq 1$, let
$$
M_{(n)}=\{m_{ikj}\}_{1\leq i,j\leq n}\quad\hbox{and}\quad
M_{(n)}(k)=\{m_{ikj}\}_{1\leq i,k,j\leq n}
$$
be the standard truncations of these matrices.

We may identify
$\mathcal S^2_n$ (respectively $\mathcal S^\infty_n$) with the subspace
of $\mathcal S^2$ (respectively $\mathcal S^\infty$) spanned by
$\{E_{ij}\, :\, 1\leq i,j\leq n\}$. Then the union $\cup_{n\geq 1}\mathcal S^2_n$
is dense in $\mathcal S^2$. Hence by a standard density argument,
$M$ is a bilinear Schur multiplier into
$\mathcal S^1$ if and only if
$\sup_{n\geq 1}\Vert M_{(n)}:\mathcal S^2_n\times \mathcal S^2_n\to
\mathcal S^1_n\bigr\Vert\,<\infty$, and in this case
$$
\bigl\Vert M:\mathcal S^2\times \mathcal S^2\to
\mathcal S^1\bigr\Vert\, =\,
\sup_{n\geq 1}\,\bigl\Vert M_{(n)}:\mathcal S^2_n\times \mathcal S^2_n\to
\mathcal S^1_n\bigr\Vert.
$$
Likewise $\cup_{n\geq 1}\mathcal S^\infty_n$
is dense in the space $\mathcal S^\infty$ of all compact operators,
for any $k\geq 1$
$M(k)$ is a linear Schur multiplier on $\mathcal
S^\infty$ if and only if
$\sup_{n\geq 1}\Vert M_{(n)}(k):\mathcal S^\infty_n \to
\mathcal S^\infty_n\bigr\Vert\,<\infty$, and
$$
\bigl\Vert M(k) :\mathcal S^\infty \to
\mathcal S^\infty\bigr\Vert\, =\,
\sup_{n\geq 1}\bigl\Vert M_{(n)}(k) :\mathcal S^\infty_n
\to \mathcal S^\infty_n\bigr\Vert.
$$
in this case.

Combining the above two approximation results with Theorem
\ref{main theorem_finite}, we obtain the result.
\end{proof}

Theorem \ref{main theorem} together with Theorem \ref{Multipliers on S_infty} yield the following result.

\begin{corollary}\label{main-factor}
A three-dimensional matrix $M=\{m_{ikj}\}_{i,k,j\ge 1}$
is a bilinear Schur multiplier into
$\mathcal S^1$ if and only if there exist a Hilbert space $E$
and two bounded families
$(\xi_{ik})_{i,k\geq 1}$ and $(\eta_{jk})_{j,k\geq 1}$ in $E$ such that
$$
m_{ikj} = \langle \xi_{ik},\eta_{jk}\rangle,
\qquad i,k,j\geq 1.
$$
Moreover
$$
\bigl\Vert M:\mathcal S^2 \times \mathcal S^2 \to \mathcal S^1\bigr\Vert\,=\,
\inf\bigl\{\sup_{i,k}\Vert \xi_{ik}\Vert\,
\sup_{j,k}\Vert \eta_{jk}\Vert\bigr\},
$$
where the infimum runs over all possible such factorizations.
\end{corollary}

\section{Schur multipliers associated with a function and self-adjoint operators}\label{sec_Connection}

Throughout this section we work with finite-dimensional operators. We fix an integer
$n\geq 1$ and regard  $\mathbb{C}^n$ as equipped with its standard Hermitian structure.

Consider two orthonormal bases $e=\{e_j\}_{j=1}^n$ and $e'=\{e'_i\}_{i=1}^n$ in $\mathbb C^n$.
Then every linear
operator $A\in B(\mathbb{C}^n)$ is associated with a matrix
$A=\{a_{ij}\}_{i,j= 1}^n,$ where $a_{ij}=\langle A(e_j), e'_i\rangle.$
Sometimes we use the notation $a_{ij}^{e',e}$ instead of
$a_{ij}$ to emphasize corresponding bases.

For any  unit vector $x\in \mathbb{C}^n$ we let $P_x$
denote the projection on the  linear span of $x,$ that is,
$P_x(y)=\langle y,x\rangle x,$  $y\in \mathbb{C}^n.$

\subsection{Linear Schur multipliers}\label{LSM}

 Let $A_0,A_1\in B(\mathbb C^n)$ be diagonalizable
self-adjoint operators. For $j=0,1$, let $\xi_j=\{\xi_i^{(j)}\}_{i=1}^n$
be an orthonormal basis of eigenvectors for $A_j$, and let
$\{\lambda_i^{(j)}\}_{i=1}^n$ be the associated
$n$-tuple of eigenvalues, that is, $A_j(\xi_i^{(j)})=\lambda_i^{(j)}\xi_i^{(j)}$.
Without loss of generality, we assume that $\{\lambda_i^{(j)}\}_{i=1}^{n_j}$ is the set
of pairwise distinct eigenvalues of the operator $A_j$,
where $n_j\in\mathbb N,$ $n_j\le n$. Denote
\begin{equation}\label{spectral}
E^{(j)}_i=\sum_{\substack{ k=1\\ \lambda_k^{(j)}=
\lambda_i^{(j)}}}^n P_{\xi_k^{(j)}}, \quad 1\le i\le n_j,
\end{equation}
that is, $E^{(j)}_i$ is a spectral projection
of the operator $A_j$ associated with the eigenvalue
$\lambda_i^{(j)}$.

Let $\phi:\mathbb R^2\to \mathbb C$ be a bounded Borel function.
Define a linear operator
$T_\phi^{A_0,A_1}:B(\mathbb C^n)\to B(\mathbb C^n)$ given by
\begin{equation}\label{linearMOI}
T_\phi^{A_0,A_1}(X)=\sum_{i,k=1}^{n} \phi(\lambda_i^{(0)},
\lambda_k^{(1)})P_{\xi_i^{(0)}}XP_{\xi_k^{(1)}}, \
\ X\in B(\mathbb C^n).
\end{equation}
Alternatively, when it is more convenient, we will
use the representation of $T_\phi^{A_0,A_1}(X)$ in the form
\begin{equation}\label{linearMOI_in spectral proj}
T_\phi^{A_0,A_1}(X)
=\sum_{i=1}^{n_0} \sum_{k=1}^{n_1}  \phi(\lambda_i^{(0)},
\lambda_k^{(1)})E^{(0)}_iXE^{(1)}_k, \ \ X\in B(\mathbb C^n).
\end{equation}

It is not difficult to see that if we identify $B(\mathbb C^n)$ with $M_n$
by associating $X$ with the matrix
$\{x_{ik}^{\xi_0,\xi_1}\}_{i,k=1}^n$,
then the operator $T_\phi^{A_0,A_1}$ acts as a linear Schur multiplier
$\{\phi(\lambda_i^{(0)}, \lambda_k^{(1)})\}_{i,k=1}^n$.  Indeed,
$$
\bigl\langle (P_{\xi_i^{(0)}}XP_{\xi_k^{(1)}})(\xi_s^{(1)}),\xi_r^{(0)}\bigr\rangle =
\left\{\begin{array}{cl}
\langle X(\xi_s^{(1)}),\xi_r^{(0)}\rangle=
x_{rs}^{\xi_0,\xi_1},& \text{if} \ \ s=k,\, r=i, \\
0 & \text{otherwise.}\end{array}\right.
$$
Therefore,
$$
\bigl\langle T_\phi^{A_0,A_1}(X)(\xi_k^{(1)}),\xi_i^{(0)}\bigr\rangle
= \phi(\lambda_i^{(0)},
\lambda_k^{(1)}) x_{ik}^{\xi_0,\xi_1},
$$
which implies that $T_\phi^{A_0,A_1}\sim \{\phi(\lambda_i^{(0)},
\lambda_k^{(1)})\}_{i,k=1}^n\colon M_n \to M_n$.
Since these identifications
are isometric ones, we deduce that
\begin{equation}\label{normT1}
\Vert T_\phi^{A_0,A_1}\colon \mathcal S^\infty_n \to \mathcal S^\infty_n\Vert = \bigl\Vert
\{\phi(\lambda_i^{(0)},
\lambda_k^{(1)})\}_{i,k=1}^n\colon \mathcal S^\infty_n \to \mathcal S^\infty_n\bigr\Vert.
\end{equation}

The operator $T_\phi^{A_0,A_1}$ is called a linear
Schur multiplier associated with $\phi$ and $A_0,A_1.$

\subsection{Bilinear Schur multipliers}

Similarly, we introduce bilinear Schur multipliers
 associated to a triple of
self-adjoint operators.

Let $A_0, A_1, A_2\in B(\mathbb C^n)$ be diagonalizable
self-adjoint operators and for any $j=0,1,2$, let
$\xi_j=\{\xi_i^{(j)}\}_{i=1}^n$ be an orthornomal
basis of eigenvectors of $A_j$ and let $\{\lambda_i^{(j)}\}_{i=1}^n$
be the corresponding $n$-tuple of eigenvalues.

Let $\psi:\mathbb R^3\to \mathbb C$ be a  bounded Borel function.
Define a bilinear operator
$T_\psi^{A_0,A_1,A_2}:B(\mathbb C^n)\times B(\mathbb C^n)\to B(\mathbb C^n)$ by setting
\begin{equation}\label{bilinearMOI}
T_\psi^{A_0,A_1,A_2}(X,Y)=\sum_{i,j,k=1}^n \psi(\lambda_i^{(0)},
\lambda_k^{(1)},\lambda_j^{(2)})P_{\xi_i^{(0)}}XP_{\xi_k^{(1)}}
YP_{\xi_j^{(2)}}
\end{equation}
for any $X,Y\in B(\mathbb C^n).$  Assume that $\{\lambda_i^{(j)}\}_{i=1}^{n_j}$ is the
set of pairwise distinct eigenvalues of the operator $A_j$. Then alternatively, using
the spectral projections (\ref{spectral}), we can write
\begin{equation}\label{bilinearMOI_in spectral proj}
T_\psi^{A_0,A_1,A_2}(X,Y)=
\sum_{i=1}^{n_0} \sum_{k=1}^{n_1}\sum_{j=1}^{n_2}
\psi(\lambda_i^{(0)},
\lambda_k^{(1)},\lambda_j^{(2)})E^{(0)}_i X E^{(1)}_k
Y E^{(2)}_j
\end{equation}
for any $X,Y\in B(\mathbb C^n).$

Let us consider two different identifications of $B(\mathbb C^n)$ with $M_n$.
On the one hand, we identify $X$ with the matrix
$\{x_{ik}^{\xi_0,\xi_1}\}_{i,k=1}^n$, where
$x_{ik}^{\xi_0,\xi_1}=\langle X(\xi_k^{(1)}), \xi_i^{(0)}\rangle$.
On the other hand we identify $Y$ with
$\{y_{kj}^{\xi_1,\xi_2}\}_{k,j=1}^n,$ where
$y_{kj}^{\xi_1,\xi_2}=\langle Y(\xi_j^{(2)}), \xi_k^{(1)}\rangle$.
Under these identifications,
the operator $T_\psi^{A_0,A_1,A_2}$ acts as a bilinear Schur multiplier
associated with the matrix
$M=\{\psi(\lambda_i^{(0)}, \lambda_k^{(1)},\lambda_j^{(2)})\}_{i,j,k=1}^n$.
Indeed,
$$
\bigl\langle (P_{\xi_i^{(0)}}XP_{\xi_k^{(1)}} YP_{\xi_j^{(2)}})(\xi_s^{(2)}),\xi_r^{(0)}
\bigr\rangle =
\bigl\langle Y(\xi_{s}^{(2)}),\xi_k^{(1)}\bigr\rangle \bigl\langle X(\xi_k^{(1)}),\xi_r^{(0)}
\bigr\rangle=
y_{ks}^{\xi_1,\xi_2}x_{rk}^{\xi_0,\xi_1}
$$
if $s=j, r=i$, and
$$
\bigl\langle  (P_{\xi_i^{(0)}}XP_{\xi_k^{(1)}} YP_{\xi_j^{(2)}})(\xi_s^{(2)}),\xi_r^{(0)}
\bigr\rangle =0
$$
otherwise.

Therefore,
$$
\bigl\langle T_\psi^{A_0,A_1,A_2}(X,Y)(\xi_s^{(2)}),\xi_r^{(0)}\bigr\rangle=
\sum_{k=1}^n \psi(\lambda_r^{(0)}, \lambda_k^{(1)},
\lambda_s^{(2)})y_{ks}^{\xi_1,\xi_2}x_{rk}^{\xi_0,\xi_1}, $$
which implies
$$
T_\psi^{A_0,A_1,A_2}(X,Y)=\sum_{i,j,k=1}^n \psi(\lambda_i^{(0)},
\lambda_k^{(1)},\lambda_j^{(2)})x_{ik}^{\xi_0,\xi_1} y_{kj}^{\xi_1,\xi_2}E_{ij}^{\xi_0,\xi_2}.
$$

Since these identifications are isometric ones with respect to all Schatten norms, 
we deduce the formula
\begin{equation}\label{normT2}
\bigl\Vert T_\psi^{A_0,A_1,A_2}\colon \mathcal S^2_n\times
\mathcal S^2_n\to \mathcal S^1_n \bigr\Vert\,=\,
\bigl\Vert \{\psi(\lambda_i^{(0)}, \lambda_k^{(1)},\lambda_j^{(2)})\}_{i,j,k=1}^n
\colon \mathcal S^2_n\times
\mathcal S^2_n\to \mathcal S^1_n \bigr\Vert.
\end{equation}

The operator $T_\psi^{A_0,A_1,A_2}$ is called a bilinear Schur multiplier associated 
with $\psi$ and the operators $A_0,A_1,A_2$.

Operators $T_\psi^{A_0,A_1,A_2}$ present a special case of what is known in 
the literature as ``multiple operator integrals". We refer to
 \cite{Pav,St, Peller2006,ACDS,PSS-SSF} for additional information on this notion.

\subsection{A few properties of Schur multipliers}

In this subsection, $\phi\colon{\mathbb R}^2\to\mathbb C\,$ and
$\psi\colon{\mathbb R}^3\to\mathbb C\,$ denote arbitrary bounded
Borel functions, and $n\in\mathbb N$ is a fixed integer.
The following lemma gives some nice properties of bilinear Schur multipliers.

\begin{lemma}\label{MOIprop}Let $A_0,A_1,A_2\in B(\mathbb C^n)$ be self-adjoint operators.
Let $I_n$ be the identity operator in $B(\mathbb C^n).$ Then for $j=0,1$ we have
\begin{itemize}
\item[$(i)$]
$$
T^{A_0,A_1,A_2}_\psi(A_j,X)=T^{A_0,A_1,A_2}_{\psi_j}(I_n,X),
\quad  X\in B(\mathbb C^n),
$$
where $$\psi_j(x_0,x_1,x_2)=x_j\psi(x_0,x_1,x_2), \ \ x_0,x_1,x_2\in\mathbb R.$$
\item[$(ii)$]
$$
T^{A_j,A_2}_\phi(X)=T^{A_0,A_1,A_2}_{\tilde\psi_j}(I_n,X), \quad
X\in B(\mathbb C^n),
$$
where $$\tilde\psi_j(x_0,x_1,x_2)=\phi(x_j,x_2), \ \ x_0,x_1,x_2\in\mathbb R.$$
\end{itemize}
\end{lemma}
\begin{proof}
Let us prove the assertion for $j=0$ only. The proof for $j=1$ is similar.

$(i).$
For $X\in B(\mathbb C^n)$ we have
\begin{align*}
T^{A_0,A_1,A_2}_\psi(A_0,X)&=
\sum_{i,j,k=1}^n \psi(\lambda_i^{(0)}, \lambda_k^{(1)},\lambda_j^{(2)})
P_{\xi_i^{(0)}}A_0P_{\xi_k^{(1)}} XP_{\xi_j^{(2)}}\\ &=\sum_{i,j,k=1}^n
\psi(\lambda_i^{(0)}, \lambda_k^{(1)},\lambda_j^{(2)})P_{\xi_i^{(0)}}
\Big(\sum_{r=1}^n \lambda_r^{(0)}P_{\xi_r^{(0)}}\Big)P_{\xi_k^{(1)}}
XP_{\xi_j^{(2)}}\\&=\sum_{i,j,k=1}^n \lambda_i^{(0)}\psi(\lambda_i^{(0)},
\lambda_k^{(1)},\lambda_j^{(2)})P_{\xi_i^{(0)}} I_n P_{\xi_k^{(1)}}
XP_{\xi_j^{(2)}}\\& =T^{A_0,A_1,A_2}_{\psi_0}(I_n,X).
\end{align*}
$(ii).$
For $X\in B(\mathbb C^n)$ we have
\begin{align*}
T^{A_0,A_1,A_2}_{\tilde\psi_0}(I_n,X)&=
\sum_{i,j,k=1}^n \tilde\psi_0(\lambda_i^{(0)},\lambda_k^{(1)},
\lambda_j^{(2)})P_{\xi_i^{(0)}} I_n P_{\xi_k^{(1)}} XP_{\xi_j^{(2)}}\\
&=\sum_{i,j=1}^n \phi(\lambda_i^{(0)}, \lambda_j^{(2)})P_{\xi_i^{(0)}}
\Big(\sum_{k=1}^n P_{\xi_k^{(1)}}\Big)XP_{\xi_j^{(2)}}\\&=
\sum_{i,j=1}^n \phi(\lambda_i^{(0)},\lambda_j^{(2)})P_{\xi_i^{(0)}}
XP_{\xi_j^{(2)}}\\& =T^{A_0,A_2}_{\phi}(X).
\end{align*}
\end{proof}

\begin{lemma}\label{tech_lem1}
Let $A\in B(\mathbb C^{n})$ be a self-adjoint operator and $X,Y\in B(\mathbb C^{n}).$
Let
$$
\tilde A=\left(\begin{array}{cc}
A & 0 \\ 0 & A
\end{array}\right)\quad\hbox{and}\quad \tilde X=\left(\begin{array}{cc}
0 & X \\ Y & 0
\end{array}\right).
$$
Then
$$
T^{\tilde A,\tilde A,\tilde A}_{\psi}(\tilde X,\tilde X)=\left(\begin{array}{cc}
T^{A,A,A}_{\psi}(X,Y) & 0 \\ 0 & T^{A,A,A}_{\psi}(Y,X)
\end{array}\right).
$$
\end{lemma}

\begin{proof}
Let $\{\lambda_i\}_{i=1}^m$ be the set of
distinct eigenvalues of the operator $A,$ $m\le n,$
and let $E_i^A$ be the spectral projection of $A$
associated with $\lambda_i,$ $1\le i\le m.$
Clearly, the operator $\tilde A$ has the same set
$\{\lambda_i\}_{i=1}^m$ of distinct eigenvalues and the
spectral projection of the operator $\tilde A$ associated with $\lambda_i$ is given by
$$
E_i^{\tilde A}=\left(\begin{array}{cc}
E_i^A & 0 \\ 0 & E_i^A
\end{array}\right), \ \ 1\le i\le m.
$$
Therefore, we have
\begin{align*}
T^{\tilde A,\tilde A,\tilde A}_{\psi}(\tilde X,\tilde X)
& =\sum_{i,k,j=1}^m\psi(\lambda_i,\lambda_k,\lambda_j)
\left(\begin{array}{cc}
E_i^A & 0 \\ 0 & E_i^A
\end{array}\right)\left(\begin{array}{cc}
0 & X \\ Y & 0
\end{array}\right)\times\\
&
\qquad\qquad \qquad\qquad\qquad\qquad\left(\begin{array}{cc}
E_k^A & 0 \\ 0 & E_k^A
\end{array}\right)\left(\begin{array}{cc}
0 & X \\ Y & 0
\end{array}\right)\left(\begin{array}{cc}
E_j^A & 0 \\ 0 & E_j^A
\end{array}\right)\\
&
=\sum_{i,k,j=1}^m\psi(\lambda_i,\lambda_k,\lambda_j)
\left(\begin{array}{cc}
E_i^A X E_k^A Y E_j^A & 0 \\ 0 & E_i^A Y E_k^A X E_j^A
\end{array}\right)\\
&
=\left(\begin{array}{cc}
T^{A,A,A}_{\psi}(X,Y) & 0 \\ 0 & T^{A,A,A}_{\psi}(Y,X)
\end{array}\right).
\end{align*}
\end{proof}

\begin{lemma}\label{tech_lem2}
Let $A, B\in B(\mathbb C^{n})$ be self-adjoint
operators with the same set of eigenvalues and
$X,Y\in B(\mathbb C^{n}).$
Let
$$
\tilde A=\left(\begin{array}{cc}
A & 0 \\ 0 & B
\end{array}\right), \quad \tilde X=\left(\begin{array}{cc}
0 & X \\ 0 & 0
\end{array}\right)\quad\hbox{and}\quad
\tilde Y=\left(\begin{array}{cc}
0 & 0 \\ 0 & Y
\end{array}\right).
$$
Then
$$T^{\tilde A,\tilde A,\tilde A}_{\psi}(\tilde X,\tilde Y)=\left(\begin{array}{cc}
0 & T^{A,B,B}_{\psi}(X,Y) \\ 0 &0
\end{array}\right).$$
\end{lemma}

\begin{proof}
Let $\{\lambda_i\}_{i=1}^m$ be the set of
distinct eigenvalues of the operator $A,$ $m\le n,$
and let $E_i^A$ (resp. $E_i^B$) be the spectral projection of $A$
(resp. $B$) associated with $\lambda_i,$ $1\le i\le m.$
Since $A$ and $B$ have the same set of eigenvalues, the operator
$\tilde A$ has the same set
$\{\lambda_i\}_{i=1}^m$ of distinct eigenvalues and the spectral
projection of the operator $\tilde A$ associated with $\lambda_i$ is given by
$$
E_i^{\tilde A}=\left(\begin{array}{cc}
E_i^A & 0 \\ 0 & E_i^B
\end{array}\right), \ \ 1\le i\le m.
$$
Therefore, we have
\begin{align*}
T^{\tilde A,\tilde A,\tilde A}_{\psi}(\tilde X,\tilde Y)&  =
\sum_{i,k,j=1}^m\psi(\lambda_i,\lambda_k,\lambda_j)
\left(\begin{array}{cc}
E_i^A & 0 \\ 0 & E_i^B
\end{array}\right)\left(\begin{array}{cc}
0 & X \\ 0 & 0
\end{array}\right)\times\\
&\qquad\qquad \qquad\qquad\qquad\qquad
\left(\begin{array}{cc}
E_k^A & 0 \\ 0 & E_k^B
\end{array}\right)
\left(\begin{array}{cc}
0 & 0 \\ 0 & Y
\end{array}\right)\left(\begin{array}{cc}
E_j^A & 0 \\ 0 & E_j^B
\end{array}\right)\\
&
=\sum_{i,k,j=1}^m\psi(\lambda_i,\lambda_k,\lambda_j)
\left(\begin{array}{cc}
0 & E_i^A X E_k^B Y E_j^B \\ 0 & 0
\end{array}\right)\\
&
=\left(\begin{array}{cc}
0 & T^{A,B,B}_{\psi}(X,Y) \\ 0 & 0
\end{array}\right).
\end{align*}
\end{proof}

\begin{lemma}\label{lem_homogen_symbol}
Let $A_0,A_1,A_2\in B(\mathbb C^{n})$ be self-adjoint operators.
For any $a\neq 0\in\mathbb R$ we have that
$$T^{aA_0,aA_1,aA_2}_{\psi}= T^{A_0,A_1,A_2}_{\psi_a},$$
where
$$\psi_a(x_0,x_1,x_2)=\psi(ax_0,ax_1,ax_2), \ \ x_0,x_1,x_2\in\mathbb R.$$
\end{lemma}

\begin{proof}
Let $\{\lambda_i^{(j)}\}_{i=1}^{n_j}$ be the set of distinct eigenvalues
of $A_j$, $j=0,1,2$. Fix $a\neq 0\in\mathbb R$. It is clear that for any $j$,
$\{a\lambda_i^{(j)}\}_{i=1}^{n_j}$ is the set of distinct eigenvalues
of $aA_j$, and that the corresponding spectral projections coincide,
that is, $E_{i}^{aA_j} = E_{i}^{A_j}$ for any $i=1,\ldots,n_j$.
Therefore, for $X,Y\in B(\mathbb C^n)$, we have
\begin{align*}
T_\psi^{a A_0,a A_1,a A_2}(X,Y)
&
=\sum_{i=1}^{n_0} \sum_{k=1}^{n_1}\sum_{j=1}^{n_2}
\psi\bigl(a\lambda_i^{(0)}, a\lambda_k^{(1)},a \lambda_j^{(2)}\bigr)E_{i}^{A_0}
X E_{k}^{A_1} Y E_{j}^{A_2}\\
&
=T^{A_0,A_1,A_2}_{\psi_a}(X,Y).
\end{align*}
\end{proof}

\begin{lemma}\label{lem_contin_symbol}
Let $A, B\in B(\mathbb C^{n})$  be self-adjoint operators and let $\{U_m\}_{m\ge 1}$
be a sequence of unitary operators from $B(\mathbb C^{n})$ such that
$U_m\to I_n$ as $m\to \infty$. Let also
$X,Y\in B(\mathbb C^{n})$ and sequences $\{X_m\}_{m\ge 1}$
and $\{Y_m\}_{m\ge 1}$ in $B(\mathbb C^{n})$  such that
$X_m\to X$ and $Y_m\to Y$ as $m\to \infty$.
Let $\psi,\psi_m: \mathbb R^3 \to \mathbb C$
be bounded Borel functions such that $\psi_m\to\psi$ pointwise
as $m\to\infty$.
Then
\begin{equation}
\label{conv}T^{U_m A U_m^*,B,B}_{\psi_m}(X_m,Y_m)\longrightarrow
T^{A,B,B}_{\psi}(X,Y), \ \ m\to\infty.
\end{equation}
\end{lemma}

\begin{proof}
Let $\{\lambda_i\}_{i=1}^{m_0}$ and $\{\mu_k\}_{k=1}^{m_1}$ be the set of
distinct eigenvalues of the operators $A$ and $B$, respectively, $m_0,m_1\le n,$
and let $E_i^A$ (resp. $E_k^B$) be the spectral projection of $A$ (resp. $B$)
associated with $\lambda_i$ (resp. $\mu_k$), $1\le i\le m_0$
(resp. $1\le k\le m_1$).
It is clear that the sequence $\{\lambda_i\}_{i=1}^{m_0}$ is the sequence
of eigenvalues of $U_m A U_m^*$
and that the spectral projection of $U_m A U_m^*$ associated with $\lambda_i$ is
given by 
$$
E_i^{U_m A U_m^*}=U_m E_i^{A} U_m^*,\quad 1\le i\le m_0.
$$
Observe that
\begin{align*}
T^{U_m A U_m^*,B,B}_{\psi_m}(X_m,Y_m)
&
=\sum_{i=1}^{m_0}\sum_{j,k=1}^{m_1} \psi_m(\lambda_i, \mu_k, \mu_j)
E_i^{U_m A U_m^*}X E_k^{B} YE_j^{B}\\
&
= U_m\Big(\sum_{i=1}^{m_0}\sum_{j,k=1}^{m_1} \psi_m(\lambda_i,
\mu_k, \mu_j)E_i^{ A} (U_m^* X) E_k^{B} YE_j^{B}\Big)\\
&
=U_m T^{A,B,B}_{\psi_m}(U_m^* X,Y).
\end{align*}

We claim that
$T^{A,B,B}_{\psi_m}(U_m^* X,Y)\to T^{A,B,B}_{\psi}(X,Y)$.
Indeed, we have
\begin{align*}
\| & T^{A,B,B}_{\psi_m}(U_m^* X,Y)
- T^{A,B,B}_{\psi}(X,Y)\|_\infty \\
&
\le \| T^{A,B,B}_{\psi_m}(U_m^* X,Y)-T^{A,B,B}_{\psi_m}(X,Y)\|_\infty
+\|T^{A,B,B}_{\psi_m}(X,Y)- T^{A,B,B}_{\psi}(X,Y)\|_\infty\\
&
\leq \| T^{A,B,B}_{\psi_m}(U_m^* X-X,Y)\|_\infty
+\|T^{A,B,B}_{\psi_m-\psi}(X,Y)\|_\infty\\
&
\leq
\sum_{i=1}^{m_0}\sum_{j,k=1}^{m_1} |\psi_m(\lambda_i, \mu_k, \mu_j)| \|U_m X-X\|_\infty\|Y\|_\infty
\,+\\
& \qquad\qquad\qquad\qquad
\sum_{i=1}^{m_0}\sum_{j,k=1}^{m_1} |\psi_m-\psi|(\lambda_i, \mu_k, \mu_j)\|X\|_\infty\|Y\|_\infty.
\end{align*}
This upper bound tends to $0$ as $m\to\infty$, which proves the claim.

Now since $U_m\to I_n$, we have
$$
U_m T^{A,B,B}_{\psi_m}(U_m^* X,Y)-T^{A,B,B}_{\psi_m} (U_m^* X,Y)\,\longrightarrow 0
$$
as $m\to\infty$. The result follows at once.
\end{proof}

\begin{lemma}\label{lem_commuts}
Let $A\in B(\mathbb C^n)$ be a self-adjoint operator and
let $X\in B(\mathbb C^n)$ commute with $A.$
Let $\widehat{\psi}\colon\mathbb R\to \mathbb R$ be defined by
$\widehat{\psi}(x)=\psi(x,x,x)$, $x\in\mathbb R$.
\begin{itemize}
\item[$(i)$] We have
$$
 T^{A,A,A}_{\psi}(X,X)=\widehat{\psi}(A)\times X^2.
$$
\item[$(ii)$] We have
$$
T^{A,A,A}_\psi(Y,X)=T^{A,A}_{\phi_1}(Y)\times X, \ \  Y\in B(\mathbb C^n),
$$
where
$$
\phi_1(x_0,x_1)=\psi(x_0,x_1,x_1), \ \ x_0,x_1\in\mathbb R.
$$
\item[$(iii)$] We have
$$
T^{A,A,A}_\psi(X,Y)=X\times T^{A,A}_{\phi_2}(Y), \ \  Y\in B(\mathbb C^n),
$$
where
$$
\phi_2(x_0,x_1)=\psi(x_0,x_0,x_1), \ \ x_0,x_1\in\mathbb R.
$$
\end{itemize}
\end{lemma}

\begin{proof}
Let $\{\xi_i\}_{i=1}^n$ be an orthonormal basis of eigenvectors
of $A$ and let $\{\lambda_i\}_{i=1}^n$ be the associated $n$-tuple of eigenvalues.
Since $A$ commutes with $X$, it follows that the projection
$P_{\xi_i}$ commutes with $X$ for all $1\le i\le n.$ Thus,
we have that
\begin{align*}
T_{\psi}^{A,A,A}(X,X)
&
=\sum_{i,j,k=1}^n \psi(\lambda_i,
\lambda_k,\lambda_j)P_{\xi_i}XP_{\xi_k} XP_{\xi_j}\\
&
=\sum_{i=1}^n \psi(\lambda_i, \lambda_i,\lambda_i)P_{\xi_i} \times X^2\\
&
=\sum_{i=1}^n \widehat{\psi}(\lambda_i)P_{\xi_i} \times X^2\,=\,\widehat{\psi}(A)\times X^2,
\end{align*}
which proves $(i).$

Similarly, for $(ii),$ we have
\begin{align*}
T_{\psi}^{A,A,A}(Y,X)
&
=\sum_{i,j,k=1}^n \psi(\lambda_i, \lambda_k,\lambda_j)P_{\xi_i}
YP_{\xi_k} XP_{\xi_j}\\
&
=\sum_{i,k=1}^n \psi(\lambda_i, \lambda_k,\lambda_k)
P_{\xi_i}YP_{\xi_k} \times X\\
&
=\sum_{i,k=1}^n \phi_1(\lambda_i, \lambda_k)
P_{\xi_i}YP_{\xi_k} \times X=T^{A,A}_{\phi_1}(Y)\times X.
\end{align*}
The proof of $(iii)$ repeats that of $(ii).$
\end{proof}

\subsection{Divided differences}\label{DD}
Let $f\colon\mathbb R \to \mathbb R$ be a continuous function
and assume that $f$ admits right and left derivatives $f'_r(x)$
and $f'_l(x)$ at each $x\in\mathbb R$. Assume further that
$f'_r,f'_l$ are bounded.
The divided difference of the first order
is defined by
\begin{align*}
{f^{[1]} \left(x_0,x_{1} \right)} :=
\begin{cases}\frac
{ f(x_0) - f(x_1)}{x_0 - x_1}, & \text{if~$x_0
\neq x_1$} \\
\frac{f'_r(x_0)+f'_l(x_0)}{2} & \text{if~$x_0=x_1$}
\end{cases}, \ \ x_0, x_1\in\mathbb{R}.
\end{align*}
Then $f^{[1]}$ is a bounded Borel function.

Let $A_0,A_1$ as in Subsection \ref{LSM}.
We study below the multiplier $T_{f^{[1]}}^{A_0,A_1}$ and give
the formula from \cite[Theorem 5.3]{ACDS}
in the setting of matrices
(see \eqref{f(A)-f(B)_finite} below).
The symbol $f^{[1]}$ and the corresponding Schur multiplier
were first studied by L\"{o}wner in
\cite{L1934}, where he noted that since
$$
f(A_j)\xi_i^{(j)}=f(\lambda_i^{(j)})\xi_i^{(j)}, \ \ \ 1\le i\le n, \ \ j=0,1,
$$
we have
\begin{equation}\label{Lowner's
formula}
\bigl\langle(f(A_0)-f(A_1))(\xi_k^{(1)}),\xi_i^{(0)}
\bigr\rangle
=f^{[1]}(\lambda_i^{(0)},
\lambda_k^{(1)})
\bigl\langle (A_0-A_1)(\xi_k^{(1)}),\xi_i^{(0)}\bigr\rangle.
\end{equation}
Formula \eqref{Lowner's formula} implies that
\begin{equation}\label{f(A)-f(B)_finite}
f(A_0)-f(A_1)=T^{A_0,A_1}_{f^{[1]}}(A_0-A_1).
\end{equation}

Now assume that $f$ is a $C^2$-function, with a bounded second derivative
$f''$. The divided difference of the second order
is defined by
\begin{align}\label{[2]}
{f^{[2]} \left(x_0,x_{1},x_2 \right)} :=
\begin{cases}\frac
{f^{[1]}(x_0,x_1) - f^{[1]}(x_1,x_2)}{x_0 - x_2}, & \text{if~$x_0
\neq x_2$}, \\ \frac {d}{dx_0} f^{[1]}(x_0,x_1), & \text{if~$x_0=x_2$}
\end{cases}, \ \ x_0, x_1,x_2\in\mathbb{R}.
\end{align}
Then $f^{[2]}$ is a bounded Borel function, and this function
is symmetric in the three variables $(x_0,x_1,x_2)$.

The following result may be viewed as a higher dimensional version of
\eqref{f(A)-f(B)_finite}.

\begin{theorem}\label{perturbation theorem}
Let $f\in C^2(\mathbb R)$ and $A_0,A_1,A_2\in B(\mathbb C^n)$ be
self-adjoint operators. Then for all $X\in B(\mathbb C^n)$ we have
$$T^{A_0,A_2}_{f^{[1]}}(X)-T^{A_1,A_2}_{f^{[1]}}(X)=
T^{A_0,A_1,A_2}_{f^{[2]}}(A_0-A_1,X).$$
\end{theorem}

\begin{proof}
Let $X\in B(\mathbb C^n)$ and let
$\psi=f^{[2]}$ and $\phi=f^{[1]}$. Setting $\psi_0,\psi_1,
\tilde \psi_0,\tilde \psi_1$ as in Lemma \ref{MOIprop} $(i),$ $(ii)$, we have
\begin{align}\label{perturbation theorem_1}
(\psi_0-\psi_1)(x_0,x_1,x_2)
&
=x_0 f^{[2]}(x_0,x_1,x_2)-x_1
f^{[2]}(x_0,x_1,x_2)\nonumber \\
&
=f^{[1]}(x_0,x_2)- f^{[1]}(x_1,x_2)\\
&
=(\tilde\psi_0-\tilde\psi_1)(x_0,x_1,x_2).\nonumber
\end{align}
Therefore, by Lemma \ref{MOIprop}, we obtain
\begin{align*}
T^{A_0,A_1,A_2}_{f^{[2]}}(A_0-A_1,X)
&
=T^{A_0,A_1,A_2}_{f^{[2]}}(A_0,X)-
T^{A_0,A_1,A_2}_{f^{[2]}}(A_1,X)\\
&
\stackrel{Lem \ref{MOIprop}(i)}{=}
T^{A_0,A_1,A_2}_{\psi_0}(I_n,X)-T^{A_0,A_1,A_2}_{\psi_1}(I_n,X)\\
&
=T^{A_0,A_1,A_2}_{\psi_0-\psi_1}(I_n,X)\\ &\stackrel{\eqref{perturbation theorem_1}}{=}
T^{A_0,A_1,A_2}_{\tilde\psi_0-\tilde\psi_1}(I_n,X)\\
&
=T^{A_0,A_1,A_2}_{\tilde\psi_0}(I_n,X)-T^{A_0,A_1,A_2}_{\tilde\psi_1}
(I_n,X)\\&\stackrel{Lem \ref{MOIprop}(ii)}
{=}T^{A_0,A_2}_{f^{[1]}}(X)-T^{A_1,A_2}_{f^{[1]}}(X).
\end{align*}
\end{proof}

Let $f\in C^1(\mathbb R)$ and let $A,B\in B(\mathbb C^n)$ be self-adjoint operators.
Then the function $t\mapsto f(A+tB)$ is differentiable and
\begin{equation}\label{e2}
\frac{d}{dt}\bigl(f(A+tB)\bigr)\Big|_{t=0}
=T_{f^{[1]}}^{A,A}(B).
\end{equation}
Indeed this follows e.g. from \cite[Theorem 3.25]{Hiai}. This leads to
the following reformulation of \eqref{SecondTaylor}
in terms of bilinear Schur multipliers.

\begin{theorem}\label{Formula-Taylor} For any self-adjoint operators $A,B\in B(\mathbb C^n)$ and
any $f\in C^2(\mathbb R),$ we have
\begin{equation}\label{Taylor_MOI}
f(A+B)-f(A)-\frac{d}{dt}\bigl(f(A+tB)\bigr)\Big|_{t=0}=
T_{f^{[2]}}^{A+B,A,A}(B,B).
\end{equation}
\end{theorem}

\begin{proof}
By \eqref{f(A)-f(B)_finite}, we have that
$$
f(A+B)-f(A)=T_{f^{[1]}}^{A+B,A}(B).
$$
Combining with \eqref{e2} and applying
Theorem \ref{perturbation theorem}, we arrive at
\begin{align*}
f(A+B)-f(A)-\frac{d}{dt}\bigl(f(A+tB)\bigr)\Big|_{t=0} & =
T_{f^{[1]}}^{A+B,A}(B)-T_{f^{[1]}}^{A,A}(B)\\ &
=T_{f^{[2]}}^{A+B,A,A}(B,B).
\end{align*}
\end{proof}

\section{Finite-dimensional construction} \label{sec_FD}

In this section we  establish various estimates concerning
finite dimensional operators. The symbol ${\rm const}$ will
stand for uniform positive constants, not depending on the dimension.

Consider the function $f_0\colon\mathbb R\to\mathbb R$ defined by
$$
f_0(x)=|x|,\qquad x\in \mathbb R.
$$
The definition of $f_0^{[1]}$ given in Subsection
\ref{DD} applies to this function.

The following result is proved in \cite[Theorem 13]{Davies}.

\begin{theorem}\label{A and B exist}
For all $n\in\mathbb N$ there exist self-adjoint operators $A_{n},B_{n}\in B(\mathbb C^{2n+1})$ such that
the spectra of $A_n+B_n$ and $A_n$ coincide, $0$ is an eigenvalue of $A_n$,  and
\begin{equation}
\label{A and B exist_3}\|f_0(A_n+B_n)-f_0(A_n)\|_1\ge
{\rm const} \ \log n \|B_n\|_1.
\end{equation}
\end{theorem}

\begin{remark}\label{rem_eigenvalues of D_n}
The operator $A_n$  constructed in \cite{Davies} is a diagonal operator  defined
on $\mathbb C^{2n}$ and $0$ is not an eigenvalue of $A_n$. By changing the dimension
from $2n$ to $2n+1$ and adding a zero on the diagonal, one obtains the operator $A_n$
in Theorem \ref{A and B exist}, with $0$ in the spectrum.
\end{remark}

\begin{corollary}\label{DOI is unbounded}
For all $n\ge 1$, there exist self-adjoint operators
$A_{n},B_{n}\in B(\mathbb C^{2n+1})$ such that
the spectra of $A_n+B_n$ and $A_n$ coincide, and
$$
\big\|T_{f_0^{[1]}}^{A_{n}+B_{n},A_{n}}:\mathcal
S^\infty_{2n+1}\to \mathcal S^\infty_{2n+1} \big\|\ge {\rm const} \ \log n.$$
\end{corollary}
\begin{proof}
Take $A_{n},B_{n}\in B(\mathbb C^{2n+1})$ as in Theorem \ref{A and B exist}.
By \eqref{f(A)-f(B)_finite}, we have that
$$
T_{f_0^{[1]}}^{A_{n}+B_{n},A_{n}}(B_{n})=f_0(A_{n}+B_{n})-f_0(A_{n}).$$
By Theorem \ref{A and B exist}, we have that
\begin{equation*}\|T_{f_0^{[1]}}^{A_{n}+B_{n},A_{n}}(B_{n})\|_1  =\|f_0(A_{n}+B_{n})-f_0(A_{n})\|_1
\ge  {\rm const} \ \log n \|B_n\|_1.
\end{equation*}
Therefore,
\begin{equation*}\big\|T_{f_0^{[1]}}^{A_{n}+B_{n},A_{n}}:\mathcal S^1_{2n+1}\to \mathcal S^1_{2n+1} \big\|
\ge   {\rm const} \ \log n.
\end{equation*}
Since the operator $T_{f_0^{[1]}}^{A_{n}+B_{n},A_{n}}$ is a Schur multiplier, we obtain that
\begin{equation*}\big\|T_{f_0^{[1]}}^{A_{n}+B_{n},A_{n}}:\mathcal S^\infty_{2n+1}\to \mathcal S^\infty_{2n+1} \big\|
 \ge   {\rm const} \ \log n.
\end{equation*}
\end{proof}

Consider the function $g_0:\mathbb R\to \mathbb R$ given by
$$
g_0(x)= x|x| = xf_0(x),\qquad x\in\mathbb R.
$$
Although $g_0$ is not a $C^2$-function, one may define
$g_0^{[2]}(x_0,x_1,x_2)$ by (\ref{[2]}) whenever $x_0,x_1,x_2$ are not
equal. Let us define
$$
\psi_{0}(x_0,x_1,x_2):=\left\{\begin{array}{cc}
g_0^{[2]}(x_0,x_1,x_2), & \text{ if } \ x_0\neq x_1 \
\text{ or } \ x_1\neq x_2\\
2, & x_0=x_1=x_2>0 \\ -2, & x_0=x_1=x_2<0\\
0, & \text{ if } \ x_0=x_1=x_2\\
\end{array}\right..
$$
The function $\psi_0\colon \mathbb R^3\to\mathbb C$ is a bounded Borel function.

The following lemma relates the linear Schur multiplier for $f_0^{[1]}$ and
the bilinear Schur multiplier for $\psi_{0}.$

\begin{lemma}\label{connection between DOI and MOI}
For self-adjoint operators $A_n,B_n\in B(\mathbb C^{n})$ such that
$0$ belongs to the spectrum of $A_n,$ the inequality
\begin{equation}\label{connection between DOI and MOI_inequality}
\big\|T_{\psi_{0}}^{A_n+B_n,A_n,A_n}:\mathcal S^2_{n}\times
\mathcal S^2_{n}\to \mathcal S^1_{n} \big\|\ge
\big\|T_{f_0^{[1]}}^{A_n+B_n,A_n}:\mathcal S^\infty_{n}\to \mathcal S^\infty_{n}\big\|
\end{equation}
holds.
\end{lemma}

\begin{proof}
Let $\{\mu_k\}_{k=1}^{n}$ be the sequence of eigenvalues of the operator $A_n.$
For simplicity, we assume that $\mu_1=0$.

By formulas (\ref{normT1}) and (\ref{normT2}) and by
Theorem \ref{main theorem_finite},
we have that
\begin{equation*}
\big\|T_{\psi_{0}}^{A_n+B_n,A_n,A_n}:\mathcal S^2_{n}\times
\mathcal S^2_{n}\to \mathcal S^1_{n} \big\|=\max_{1\le k\le {n}}
\|T_{\varphi_k}^{A_n+B_n,A_n}:\mathcal S^\infty_{n}\to \mathcal S^\infty_{n}\|,
\end{equation*}
where
$$
\varphi_k(x_0,x_1):=\psi_{0}(x_0,\mu_k, x_1), \
\ x_0,x_1\in\mathbb R,  \ 1\le k\le n.
$$
In particular, we have
\begin{equation*}
\big\|T_{\psi_{0}}^{A_n+B_n,A_n,A_n}:\mathcal S^2_{n}\times
\mathcal S^2_{n}\to \mathcal S^1_{n} \big\|\ge
\|T_{\varphi_1}^{A_n+B_n,A_n}:\mathcal S^\infty_{n}\to \mathcal S^\infty_{n}\|.
\end{equation*}
It therefore suffices to check that
\begin{equation}\label{L17}
\varphi_1= f_0^{[1]}.
\end{equation}
It follows from the definitions that
$\varphi_1(0,0) = \psi_0(0,0,0) = 0 = f_0^{[1]}(0,0)$.

Consider now $(x_0,x_1)\in\mathbb R^2$ such that
$x_0\neq 0$ or $x_1\neq 0$. In that case, we have
$$
\varphi_1(x_0,x_1)=g_0^{[2]}(x_0, 0, x_1).
$$
If $x_0,x_1, 0$ are mutually distinct, then
\begin{align*}
g^{[2]}_0(x_0,0,x_1)& =\frac{g^{[1]}_0(x_0,0)-g_0^{[1]}(0,x_1)}{x_0-x_1}
=\frac{\frac{x_0f_0(x_0)-0}{x_0-0}-\frac{0-x_1f_0(x_1)}{0-x_1}}{x_0-x_1}\\
& =\frac{f_0(x_0)-f_0(x_1)}{x_0-x_1}=f^{[1]}_0(x_0,x_1).
\end{align*}
If $x_0=0$ and $x_1\neq 0,$ then
\begin{align*}
g^{[2]}_0(0,0,x_1) & =\frac{g^{[1]}_0(0,0)-g^{[1]}_0(0,x_1)}{x_0-x_1}=
\frac{g'_0(0)-\frac{0-x_1f_0(x_1)}{0-x_1}}{0-x_1}\\ &
=\frac{f_0(x_1)}{x_1}=f^{[1]}_0(0,x_1).
\end{align*}
The argument is similar, when $x_0\neq 0$ and $x_1= 0.$

Assume now that $x_0=x_1\not=0$. Then
we have
\begin{align*}
g_0^{[2]}(x_0,0,x_0)& =\frac{d}{dx}g_0^{[1]}(x,0)\Big|_{x=x_0}=
\frac{d}{dx}\Big(\frac{xf_0(x)-0}{x-0}\Big)\Big|_{x=x_0}\\ & =
f'_0(x_0)=f^{[1]}_0(x_0,x_0).
\end{align*}
This completes the proof of (\ref{L17}) and
we obtain \eqref{connection between DOI and MOI_inequality}.
\end{proof}

The following is a straightforward consequence of
Corollary \ref{DOI is unbounded} and Lemma \ref{connection between DOI and MOI}.

\begin{corollary}\label{MOI_estimate}
For every $n\ge 1$ there exist self-adjoint operators
$A_{n},B_{n}\in B(\mathbb C^{2n+1})$  such that
the spectra of $A_n +B_n$ and
$A_n$ coincide, and
$$
\big\|T_{\psi_{0}}^{A_{n}+B_{n},A_{n},A_{n}}:\mathcal S^2_{2n+1}\times
\mathcal S^2_{2n+1}\to \mathcal S^1_{2n+1} \big\|\ge {\rm const} \ \log n.
$$
\end{corollary}

We assume below that $n\ge 1$ is fixed and that $A_n, B_n$ are given by Corollary \ref{MOI_estimate}.
The purpose of the series of lemmas \ref{lemma_8}-\ref{lemma_5} below
is to prove Lemma \ref{lemma_3}, which is the final step in
the finite-dimensional
resolution of Peller's problem.
The following result follows immediately from Corollary \ref{MOI_estimate}.

\begin{lemma}\label{lemma_8}
There are operators  $X_n,Y_n\in B(\mathbb C^{2n+1})$ with
$\|X_n\|_2=\|Y_n\|_2=1$, such that
$$
\big\|T_{\psi_{0}}^{A_{n}+B_{n},A_{n},A_{n}}(X_n,Y_n) \big\|_1\ge
{\rm const} \ \log n.
$$
\end{lemma}

Let us denote
\begin{equation}\label{def_H_n}H_n:=\left(\begin{array}{cc}
A_n+B_n & 0 \\ 0 & A_n
\end{array}\right)\
\end{equation}
and consider the operator
$$
T_1 :=T_{\psi_{0}}^{H_n,H_n,H_n}\colon \mathcal S^2_{4n+2}\times
\mathcal S^2_{4n+2}\to \mathcal S^1_{4n+2}.
$$

\begin{lemma}\label{lemma_8_1}
There are operators  $\tilde X_n,\tilde Y_n\in B(\mathbb
C^{4n+2})$ with $\|\tilde X_n\|_2=\|\tilde Y_n\|_2= 1$, such that
$$\big\|T_1(\tilde X_n,\tilde Y_n) \big\|_1\ge
{\rm const} \ \log n.$$
\end{lemma}

\begin{proof}
Take
$$
\tilde X_n:=\left(\begin{array}{cc}
0 & X_n \\ 0_{2n+1} & 0
\end{array}\right), \ \ \tilde Y_n:=\left(\begin{array}{cc}
0_{2n+1} & 0 \\ 0 & Y_n
\end{array}\right),$$
where $X_n,Y_n$ are operators from Lemma \ref{lemma_8} and
$0_{2n+1}$ is the null element of  $B(\mathbb C^{{2n+1}}).$
Clearly,
$\|\tilde X_n\|_2=\|X_n\|_2=1$ and $\|\tilde Y_n\|_2=\|Y_n\|_2=1.$
It follows from Lemma \ref{tech_lem2}
and the fact that $A_n+B_n$ have the same
spectra that
$$
T_1(\tilde X_n,\tilde Y_n)=
\left(\begin{array}{cc}
0 & T_{\psi_{0}}^{A_{n}+B_{n},A_{n},A_{n}}(X_n,Y_n) \\ 0_{2n+1} & 0
\end{array}\right).$$
Therefore, by Lemma \ref{lemma_8},
$$\big\|T_1(\tilde X_n,\tilde Y_n) \big\|_1=
\big\|T_{\psi_{0}}^{A_{n}+B_{n},A_{n},A_{n}}(X_n,Y_n) \big\|_1\ge {\rm const} \ \log n.$$
\end{proof}

\begin{lemma}\label{lemma_7}
There is an operator  $S_n\in B(\mathbb C^{4n+2})$ with $\|S_n\|_2\le 1$ such that
$$\big\|T_1(S_n,S_n^*) \big\|_1\ge {\rm const} \ \log n.$$
\end{lemma}
\begin{proof}
Take the operators $\tilde X_n, \tilde Y_n\in B(\mathbb C^{4n+2})$ as in Lemma \ref{lemma_8_1}.
By the polarization identity
$$
T_1(\tilde X_n,\tilde Y_n)=\frac14\sum_{k=0}^3 i^k
T_1((\tilde X_n+i^k \tilde Y_n^*),(\tilde X_n+i^k \tilde Y_n^*)^*),
$$
we have that
$$
\|T_1(\tilde X_n,\tilde Y_n)\|_1\le \max_{0\le k\le 3}\|T_1((\tilde X_n+i^k \tilde Y_n^*),
(\tilde X_n+i^k \tilde Y_n^*)^*)\|_1.$$
Taking $k_0$ such that
$$
\|T_1((\tilde X_n+i^{k_0} \tilde Y_n^*),(\tilde X_n+i^{k_0} \tilde Y_n^*)^*)\|_1=
\max_{0\le k\le 3}\|T_1((\tilde X_n+i^k \tilde Y_n^*),(\tilde X_n+i^k \tilde Y_n^*)^*)\|_1,
$$
we set
$$
S_n:=\frac12(\tilde X_n+i^{k_0} \tilde Y_n^*).
$$
Thus, by Lemma \ref{lemma_8_1}, we have
$$
\big\|T_1(S_n,S_n^*) \big\|_1\ge
\frac14\|T_1(\tilde X_n,\tilde Y_n)\|_1\ge  {\rm const} \ \log n
$$
and
$$
\|S_n\|_2\le \frac12(\|\tilde X_n\|_2+\|\tilde Y_n\|_2)=1.
$$
\end{proof}

Let us denote
\begin{equation}\label{def_tilde_H_n}\tilde H_n:=\left(\begin{array}{cc}
H_n & 0 \\ 0 & H_n
\end{array}\right)=\left(\begin{array}{cccc}
A_n+B_n & 0 &0 &0 \\ 0 & A_n & 0& 0 \\  0 & 0 & A_n+B_n& 0 \\ 0 & 0 & 0& A_n \\
\end{array}\right), \ \ n\ge 1,
\end{equation}
and consider the operator
$$
T_2:=T_{\psi_{0}}^{\tilde H_n,\tilde H_n,\tilde H_n}\colon
\mathcal S^2_{8n+4}\times
\mathcal S^2_{8n+4}\to \mathcal S^1_{8n+4}.
$$

\begin{lemma}\label{lemma_6}
There is a self-adjoint  operator
$Z_n\in B(\mathbb C^{8n+4})$ with $\|Z_n\|_2\le 1$ such that
$$\big\|T_2(Z_n,Z_n) \big\|_1\ge {\rm const} \ \log n.$$
\end{lemma}
\begin{proof}
Consider the operator $S_n$ from Lemma \ref{lemma_7}.
Setting
$$Z_n:=\frac12\left(\begin{array}{cc}
0 & S_n \\ S_n^* & 0
\end{array}\right),$$
we have
$\|Z_n\|_2=\frac12(\|S_n\|_2+\|S_n^*\|_2)\le 1$
and by Lemma \ref{tech_lem1},
$$
T_2(Z_n,Z_n)=\frac14\left(\begin{array}{cc}
T_1(S_n,S_n^*) & 0 \\ 0 & T_1(S_n^*,S_n)
\end{array}\right).
$$
Therefore, by Lemma \ref{lemma_7}, we arrive at
\begin{align*}
\big\|T_2(Z_n,Z_n) \big\|_1 & =\frac14\big(\big\|T_1(S_n,S_n^*) \big\|_1+
\big\|T_1(S_n^*,S_n) \big\|_1\big)\\ &
\ge \frac14 \big\|T_1(S_n,S_n^*) \big\|_1\ge {\rm const} \ \log n.
\end{align*}
\end{proof}

The following decomposition principle is of independent interest. In this statement
we use the notation $[H,F]= HF -FH$ for the commutator of $H$ and $F$.

\begin{lemma}\label{lemma_4}
For any self-adjoint operators $Z,H \in B(\mathbb C^{n})$, there
are self-adjoint operators $F, G \in B(\mathbb C^{n})$
such that
$$
Z = G + i[H,F],
$$
the matrix $G$ commutes
with $H$,
and
$$ \left\| [H,F]\right\|_2 \leq
2\,\left\| Z\right\|_2 \ \ \text{and}\ \ \left\| G \right\|_2 \leq
\left\| Z \right\|_2.
$$
\end{lemma}

\begin{proof}
Let
$$
h_1, h_2, \ldots, h_m
$$
be the pairwise distinct eigenvalues of the operator~$H$ and let
\begin{equation*}
E_1, E_2, \ldots, E_m
\end{equation*}
be the associated spectral projections, so that
$$
H = \sum_{j =1}^m h_j E_j.
$$
We set
\begin{equation*}
G = \sum_{j = 1}^m E_j Z E_j \qquad\text{and}\qquad F = i\sum_{\substack{j=1 \\ j \neq k}}^m
(h_k - h_j)^{-1}\, E_j Z E_k.
\end{equation*}
Since
$$
HE_j = h_j E_j,
$$
we have
$$
[H,E_j Z E_k ]
=  \ H \times E_j Z E_k - \  E_j Z E_k \times H = \left( h_j - h_k \right)
\times E_j Z E_k.
$$
Consequently,
$$
i[H,F] = \sum_{\substack{j=1 \\ j \neq k}}^m E_j Z
E_k
$$ and hence
$$ G + i[H,F] = Z.
$$
Further $F,G$ are self-adjoint and it is clear that $[G,H]=0$.
Hence the first two claims of the lemma are proved.

Now take
$$
U_t = \sum_{j = 1}^m e^{ijt}\, E_j,\ \ t
\in [-\pi, \pi].
$$
 Then
\begin{equation*}
\int_{-\pi}^\pi U_t Z U_t^*\, \frac {\mathrm dt}{2\pi} = \sum_{j,
k = 1}^m E_j Z E_k \, \int_{-\pi}^\pi e^{i (j - k)\, t}\, \frac
{\mathrm dt}{2\pi}\, =\,\sum_{j = 1}^m E_j Z E_j\,=G.
\end{equation*}
Since $U_t$ is
unitary, we deduce that
$$
\left\| G \right\|_2 \leq \int_{-\pi}^\pi \left\| U_t Z
U_t^*\right\|_2\, \frac {\mathrm dt}{2\pi} \leq \left\|
Z\right\|_2.
$$
Moreover writing
$$
i[H,F] = Z - G
$$
we deduce that
$$
\left\|
[H,F]\right\|_2 \leq 2 \left\| Z\right\|_2.
$$
\end{proof}

\begin{lemma}\label{lemma_5}
There is a self-adjoint operator
$F_n\in B(\mathbb C^{8n+4})$ such that $\| [\tilde H_n,F_n]\|_2\le 2$ and
$$
\big\|T_2\bigl(i[ \tilde H_n,F_n],i[ \tilde H_n,F_n]\bigr) \big\|_1\ge {\rm const} \ \log n- 10.
$$
\end{lemma}

\begin{proof}
Take the operator $Z_n$ in
$B(\mathbb C^{8n+4})$
given by Lemma \ref{lemma_6}.
By Lemma \ref{lemma_4}, we may choose self-adjoint operators
$F_n$ and $G_n$ from $B(\mathbb C^{8n+4})$ such that
$$
Z_n = G_n + i[\tilde H_n,F_n],  \ \ \ [G_n,\tilde H_n] = 0,
$$
and
\begin{equation}\label{lemma_5_1} \| [ \tilde H_n,F_n]\|_2 \leq
2\,\left\| Z_n\right\|_2 ,\qquad \left\| G_n \right\|_2 \leq
\left\| Z_n \right\|_2.
\end{equation}
We compute
\begin{equation}\label{eq:Step3split}
\begin{aligned}[b]
T_2(Z_n, Z_n) &{}= T_2\Bigl( G_n + i[\tilde H_n,F_n],\,
G_n + i[\tilde H_n,F_n] \Bigr) \\
&{}= T_2\Bigl(G_n, \, G_n\Bigr) \\ &{}+ T_2
\Bigl( G_n, \, i[\tilde H_n,F_n] \Bigr) \\
&{}
+ T_2 \Bigl( i[\tilde H_n,F_n],\, G_n \Bigr) \\
&{}
+ T_2\Bigl( i[\tilde H_n,F_n],\,i[\tilde H_n,F_n] \Bigr).
\end{aligned}
\end{equation}
We shall estimate the first three summands above.
The operator $G_n$ commutes with $\tilde H_n$ hence
by the first part of Lemma \ref{lem_commuts},
$$
T_2(G_n,G_n)=\widehat{\psi_0}(\tilde{H}_n)\times G_n^2.
$$
Furthermore
$\widehat{\psi_0}(x)=2$ if $x>0$, $\widehat{\psi_0}(x)=-2$ if $x<0$
and $\widehat{\psi_0}(0)=0$. Hence
$$
\Vert \widehat{\psi_0}(\tilde{H}_n)\Vert_\infty\leq 2.
$$
This implies that
$$
\bigl\Vert T_2(G_n,G_n)\bigr\Vert_1\leq \Vert \widehat{\psi_0}(\tilde{H}_n)\Vert_\infty
\Vert G_n\Vert^2_2\leq 2\Vert Z_n\Vert_2^2\leq 2.
$$
Next applying the second and third part of Lemma \ref{lem_commuts}, we obtain
$$
T_2 \Bigl( i[\tilde H_n,F_n],\, G_n \Bigr) = i
T_{\phi_1}^{\tilde H_n, \tilde H_n} \Bigl( [\tilde H_n,F_n]\Bigr) \times G_n
$$
and
$$
T_2 \Bigl( G_n,i[\tilde H_n,F_n] \Bigr) = i \ G_n \times
T_{\phi_2}^{\tilde H_n,\tilde H_n} \Bigl( [\tilde H_n,F_n] \Bigr),
$$
where
\begin{multline*}
\phi_1(x_0,x_1) =\psi_{0}(x_0, x_1,x_1) \ \ \text{and} \ \
\phi_2(x_0,x_1) =\psi_{0}(x_0, x_0,x_1), \ \ x_0,x_1\in\mathbb R.
\end{multline*}
Observe that by  the Mean Value Theorem for divided differences (see e.g. \cite{dBoor}),
we have $\Vert \psi_0\Vert_\infty\leq 2$.
Hence $\|\phi_1\|_\infty\le 2$ and $\|\phi_2\|_\infty\le 2$, which implies
\begin{align*}
\Bigl\|T_{\phi_1}^{\tilde H_n, \tilde H_n} \Bigl( [\tilde H_n,F_n]
\Bigr) \times G_n \Bigr\|_1 &
\le \Bigl\|T_{\phi_1}^{\tilde H_n, \tilde H_n} \Bigl( [\tilde H_n,F_n]
\Bigr)\Bigr\|_2 \|G_n\|_2 \\
&
\le \|\phi_1\|_\infty \| [\tilde H_n,F_n]\|_2\|G_n\|_2\\
&
\leq 2\|\phi_1\|_\infty\|Z_n\|^2_2 \leq 4
\end{align*}
by \eqref{lemma_5_1} and Lemma \ref{lemma_6}. Similarly,
$$
\Bigl\|G_n\times T_{\phi_2}^{\tilde H_n, \tilde H_n} \Bigl( [\tilde H_n,F_n]
\Bigr) \Bigr\|_1\le 4.
$$
Combining the preceding estimates with \eqref{eq:Step3split}, we arrive at
$$
\|T_2(Z_n, Z_n)\|_1\le 10 +\Big\|
T_2\Bigl( i [\tilde H_n,F_n],\, i[\tilde H_n,F_n] \Bigr)\Big\|_1.
$$
Applying  Lemma \ref{lemma_6}, we deduce the result.
\end{proof}

\begin{lemma}\label{lemma_3}
There exists a $C^2$-function $g$ with a
bounded second derivative and there
exists $N\in\mathbb N$ such that for any sequence
$\{\alpha_n\}_{n\ge N}$ of positive real numbers there is  a sequence
of operators $\tilde B_{n}\in B(\mathbb C^{8n+4})$ such that
$\| \tilde B_{n}\|_2\le 4\alpha_n,$ for all $n\ge N,$ and
$$
\|T_{g^{[2]}}^{\tilde A_n+\tilde B_n, \tilde A_n, \tilde A_n}
(\tilde B_n,\tilde B_n)\|_1\ge {\rm const}\, \alpha_n^2  \ \log n,
\ \ n\ge N.
$$
\end{lemma}

\begin{proof}
 Changing the constant `const' in Lemma \ref{lemma_5} by half of its value, we can change the
estimate from that statement into
\begin{equation}\label{chooseN}
\big\|T_2\bigl(i[ \tilde H_n,F_n],i[ \tilde H_n,F_n]\bigr)
\big\|_1\ge {\rm const} \ \log n,\quad n\geq N,
\end{equation}
for sufficiently large $N\in\mathbb N$.

Take an  arbitrary sequence $\{\alpha_n\}_{n\ge N}$ of positive real numbers,
take the operator $F_n$ from Lemma \ref{lemma_5}
and denote
$$
\tilde F_n:=\alpha_n F_n.
$$
For any $t>0$, consider
$$
\gamma_t(\tilde H_n)=e^{it\tilde F_n} \tilde H_n
e^{-it\tilde F_n}, \quad \hbox{and} \quad V_{n,t}:=\frac{\gamma_t (\tilde H_n)-\tilde H_n}{t}.
$$
On the one hand, it follows from the identity
$\frac{d}{dt}\bigl(e^{it\tilde F_n}\bigr)\vert_{t=0} =
i\tilde F_n$ that
$$
V_{n,t}\longrightarrow i[\tilde F_n,\tilde H_n],  \ \ t\to +0.
$$
It therefore follows from Lemma \ref{lemma_5} that there is $t_1>0$ such that
\begin{equation}\label{VntEstimate}
\|V_{n,t}\|_2\le
 2\|[\tilde F_n,\tilde H_n]\|_2=2
\alpha_n\|[F_n,\tilde H_n]\|_2\le 4\alpha_n
\end{equation}
for all $t\le t_1.$
On the other hand,
\begin{equation}\label{t0}
\tilde H_n+t\, V_{n,t}=\gamma_t(\tilde H_n)\longrightarrow \tilde H_n,  \ \ t\to +0.
\end{equation}

Take a $C^2$-function $g$ such that $g(x)=g_0(x)=x|x|$ for $|x|>1$ and $g^{(j)}(0)=0,$
$j=0,1,2.$
Denote
$$
g_t(x_0,x_1,x_2):=g^{[2]}\Big(\frac{x_0}{t},\frac{x_1}{t},\frac{x_2}{t}\Big),
\ \ \ t>0, \ \ x_0,x_1,x_2\in\mathbb R.
$$
We claim that
\begin{equation}\label{g_to_f}
\lim_{t\to +0}g_t(x_0,x_1,x_2)=\psi_{0}(x_0,x_1,x_2),
\ \ x_0,x_1,x_2\in\mathbb R.
\end{equation}
To prove this claim, we first observe, using the definition of $g_0$, that
\begin{equation}\label{psi}
\psi_0\Bigl(\frac{x_0}{t}, \frac{x_1}{t},\frac{x_2}{t},\Bigr)=
\psi_{0}(x_0,x_1,x_2),
\ \ x_0,x_1,x_2\in\mathbb R,\ t>0.
\end{equation}
Next we note that for any $x\in\mathbb R$,
$$
g\Bigl(\frac{x}{t}\Bigr) = g_0\Bigl(\frac{x}{t}\Bigr)
\quad\hbox{and}\quad
g'\Bigl(\frac{x}{t}\Bigr) = g_0'\Bigl(\frac{x}{t}\Bigr)
$$
for $t>0$ small enough. For $x=0$, this follows from the
fact that by assumption, $g(0)=g'(0)=0$.
From these properties, we deduce that for any $x_0,x_1\in\mathbb R$,
$$
g^{[1]}\Bigl(\frac{x_0}{t}, \frac{x_1}{t}\Bigr) =
g^{[1]}_0\Bigl(\frac{x_0}{t}, \frac{x_1}{t}\Bigr)
$$
for $t>0$ small enough.

In turn, this implies that if $x_0\neq x_1$ or $x_1\neq x_2$, then
$$
g^{[2]}\Bigl(\frac{x_0}{t}, \frac{x_1}{t}, \frac{x_2}{t}\Bigr) =
g^{[2]}_0\Bigl(\frac{x_0}{t}, \frac{x_1}{t}, \frac{x_2}{t}\Bigr)
$$
for $t>0$ small enough. According to (\ref{psi}), this implies that
$$
g^{[2]}\Bigl(\frac{x_0}{t}, \frac{x_1}{t}, \frac{x_2}{t}\Bigr) =
\psi_0(x_0,x_1,x_2)
$$
for $t>0$ small enough.

Consider now the case when $x_0=x_1=x_2$. For any $t>0$, we have
$$
g^{[2]}\Bigl(\frac{x_0}{t}, \frac{x_0}{t}, \frac{x_0}{t}\Bigr) =
g_0''\Bigl(\frac{x_0}{t}\Bigr).
$$
If $x_0>0$, then $g_0''\bigl(\frac{x_0}{t}\bigr) = 2$
for $t>0$ small enough, and if $x_0<0$,
then $g_0''\bigl(\frac{x_0}{t}\bigr) = -2$
for $t>0$ small enough. Furthermore, $g_0''(0)=0$ by assumption.
Hence
$$
g^{[2]}\Bigl(\frac{x_0}{t}, \frac{x_0}{t}, \frac{x_0}{t}\Bigr) =
\psi_0(x_0,x_0,x_0)
$$
for $t>0$ small enough. This completes the proof of (\ref{g_to_f}).

Applying subsequently Lemma \ref{lem_homogen_symbol}
with $a=\frac1t$, property (\ref{t0}) and Lemma \ref{lem_contin_symbol}, 
we obtain that
\begin{align*}
T_{g^{[2]}}^{\frac1{t}\tilde H_n+ V_{n,t}, \frac1{t}\tilde H_n,
\frac1{t}\tilde H_n}(V_{n,t}, V_{n,t})
&
= T_{g_t}^{\tilde H_n+ tV_{n,t}, \tilde H_n, \tilde H_n}(V_{n,t}, V_{n,t})\\
&
\longrightarrow
T_2\bigl(i[\tilde F_n,\tilde H_n],i[\tilde F_n,\tilde H_n]\bigr)
\end{align*}
when $t\to +0$.
Furthermore,
$$
T_2\bigl(i[\tilde F_n,\tilde H_n],i[\tilde F_n,\tilde H_n]\bigr) =
\alpha_n^2T_2\bigl(i[F_n,\tilde H_n],i[F_n,\tilde H_n]\bigr).
$$
By
\eqref{chooseN}, there is $t_2>0$ such that
$$
\bigl\|T_{g^{[2]}}^{\frac1{t}\tilde H_n+V_{n,t}, \frac1{t}\tilde H_n, \frac1{t}\tilde H_n}(V_{n,t},
V_{n,t})\bigr\|_1\ge {\rm const} \ \alpha_n^2\ \log n
$$ for all $t\le  t_2.$
Taking $t_n=\min\{t_1,t_2\},$ and setting
$$
\tilde A_n:= \frac1{t_n}\tilde H_n, \ \ \tilde B_n:= V_{n,t_n},
$$
we obtain that $\|\tilde B_{n}\|_2\le 4 \alpha_n$ (see \eqref{VntEstimate}) and
$$\|T_{g^{[2]}}^{\tilde A_n+\tilde B_n, \tilde A_n,
\tilde A_n}(\tilde B_n,\tilde B_n)\|_1\ge {\rm const}\, \alpha_n^2 \  \log n ,$$
for all $n\ge N$.
\end{proof}

\section{Answering Peller's problem} \label{sec_PellersProblem}

Let $\{\mathcal H_n\}_{n=1}^\infty$ be a sequence of finite dimensional
Hilbert spaces and consider their Hilbertian direct sum
$$
\mathcal H = \overset{2}{\oplus}_{n\geq 1} \mathcal H_n.
$$
Let $\{A_n\}_{n=1}^\infty$ be a sequence of self-adjoint
operators, with $A_n\in B(\mathcal H_n)$.
Let $A$ denote their direct sum (notation  
$A=\oplus_{n=1}^\infty A_n$). Namely $A$ is defined on the domain
$$
D(A)=\Bigl\{\{\xi_n\}_{n=1}^\infty\in\mathcal H\, :\,
\sum_{n=1}^\infty\Vert  A_n (\xi_n)\Vert^2\,<\infty\,\Bigr\},
$$
by setting $A(\xi) = \{A_n (\xi_n)\}_{n=1}^\infty$ for any
$\xi=\{\xi_n\}_{n=1}^\infty$ in $D(A)$. Then $A$ is a
self-adjoint (possibly unbounded) operator on $\mathcal H$.

Likewise we let $\{B_n\}_{n=1}^\infty$ be a sequence of self-adjoint
operators, with $B_n\in \mathcal S^2(\mathcal H_n)$, and
we set $B=\oplus_{n=1}^\infty B_n$. Assume further that 
$\sum_{n=1}^{\infty}\Vert B_n\Vert_2^2\,<\infty$. 
Then $B\in \mathcal S^2(\mathcal H)$ and
\begin{equation}\label{Sum}
\Vert B\Vert_2^2\,=\,\sum_{n=1}^{\infty}\Vert B_n\Vert_2^2.
\end{equation}

Let $f\colon\mathbb R\to\mathbb R$ be a $C^2$-function with a bounded second derivative.
Then $f^{[2]}$ is bounded, with $\Vert f^{[2]}\Vert_\infty =  \Vert f''\Vert_\infty$. 
Hence according to Theorem \ref{Formula-Taylor} and Lemma \ref{always bounded}, we have
$$
\Bigl\Vert 
f(A_n+B_n)-f(A_n)-\frac{d}{dt}
\Big(f(A_n+tB_n)\Big)\Big|_{t=0}\Bigr\Vert_2\leq \Vert f''\Vert_\infty\Vert B_n\Vert_2^2.
$$
We deduce that
\begin{align*}
\sum_{n=1}^\infty \,\Bigl\Vert 
f(A_n+B_n)-f(A_n)-\frac{d}{dt}
\Big(f(A_n+tB_n)\Big) & \Big|_{t=0}\Bigr\Vert_2^2\\
& \leq \Vert f''\Vert_\infty^2\,\Bigl(\sum_{n=1}\Vert B_n\Vert_2^2\Bigr)^2\,<\infty.
\end{align*}
Then we may define
\begin{multline*}
f(A+B)-f(A)-\frac{d}{dt}
\Big(f(A+tB)\Big)\Big|_{t=0}\\:=\bigoplus_{n=1}^\infty\Big(f(A_n+B_n)-f(A_n)-\frac{d}{dt}
\Big(f(A_n+tB_n)\Big)\Big|_{t=0}\Big),
\end{multline*}
which is an element of $\mathcal S^2(\mathcal H)$.

We note that the above construction can be carried out as well
in the case when the $\mathcal H_n$'s are infinite 
dimensional, provided that each $A_n$ is a bounded operator.

The following theorem answers Peller's problem \eqref{MainQuestion} in negative.

\begin{theorem}
There exists a function $f\in C^2(\mathbb R)$ with a bounded second derivative,  
a self-adjoint operator $A$ on $\mathcal H$ and a self-adjoint
$B\in \mathcal S^2(\mathcal H)$ as above such that
$$
f(A+B)-f(A)-\frac{d}{dt}
\Big(f(A+tB)\Big)\Big|_{t=0} \notin\mathcal S^1.
$$
\end{theorem}

\begin{proof}
Take the integer $N\in\mathbb N$, the operators
$\tilde A_n$, $\tilde B_n$ and the function $g$ from Lemma
\ref{lemma_3}, applied with the sequence $\{\alpha_n\}_{n\geq N}$ defined by
$$
\alpha_n=\,\frac1{\sqrt{n \ \log^{3/2} n}}\,.
$$
Let $\mathcal H_n = \ell^2_{8n+4}$ and let $\mathcal H=\overset{2}{\oplus}_{n\geq N} \mathcal H_n$.
Then let $A=\oplus_{n=N}^\infty A_n$ and $B=\oplus_{n=N}^\infty B_n$ be the corresponding direct sums.
Then the self-adjoint operator $B$ belongs
to $S^2(\mathcal H)$. Indeed, it follows from (\ref{Sum}) and
Lemma~\ref{lemma_3} that
$$
\|B\|_2^2=\sum_{n=N}^\infty \|\tilde B_{n}\|_2^2\le 16
\sum_{n=N}^\infty \alpha_n^2 =
\sum_{n=N}^\infty\frac{16}{{n \ \log^{3/2} n} }<\infty.
$$

On the other hand, by \eqref{Taylor_MOI} and  Lemma \ref{lemma_3}, we have
\begin{align*}\label{series}
\Bigl\|g(A+B)- & g(A)-\frac{d}{dt}\bigl(g(A+tB)
\bigr)\Big|_{t=0}\Bigr\|_1\\
&
=\sum_{n=N}^\infty\Big\|g(\tilde A_n+\tilde B_{n})-g(\tilde A_n)-
\frac{d}{dt}\bigl(g(\tilde A+t\tilde B_{n})\bigr)\Big|_{t=0}\Big\|_1\\
&
=\sum_{n=N}^\infty
\left\| T_{g^{[2]}}^{\tilde A_n+\tilde B_{n},\tilde A_n,\tilde A_n}
\Bigl(\tilde B_{n},\, \tilde B_{n}\Bigr) \right\|_1\\
&
\geq \,{\rm const}\sum_{n=N}^\infty \alpha_n^2 \, \log n\\
&
={\rm const}
\sum_{n=N}^\infty \frac1{{n \ \log^{1/2} n} }=\infty.
\end{align*}
\end{proof}

\bigskip\noindent
\textbf{Acknowledgement}. The authors C.C and C.L have been supported by the research program ANR 2011 BS01 008 01
and by the ``Conseil r\'egional de Franche-Comt\'e".
The authors D.P., F.S and A.T. have been supported by the Australian Research Council grant DP150100920.

\end{document}